\documentclass[12pt]{amsart}

\advance\hoffset by -0.5 truecm

\usepackage{amssymb}
\usepackage{amscd}
\usepackage{verbatim}
\usepackage{epsfig}
\usepackage{euscript}
\usepackage[colorlinks=true]{hyperref}
\hypersetup{urlcolor=blue, citecolor=red}
\usepackage{color}

\newcommand{\supp}{\operatorname{supp}}

\newcommand{\diag}{\mathrm{diag}}

\def \R{\mathbb R}
\def \p{\partial}

\def \<{\langle}
\def \>{\rangle}

\def \1N{\sum_{k=1}^{N}}
\def \A{\mathcal A}

\def \F{\mathcal F}

\def \dgamma{\dot\gamma}
\def \lg{\mathfrak g}

\def \0X{X^{(0)}}

\def \X{\mathbb X}

\newtheorem{lemma}{Lemma}[section]
\newtheorem{prop}[lemma]{Proposition}
\newtheorem{thm}[lemma]{Theorem}
\newtheorem{cor}[lemma]{Corollary}

\newtheorem*{thm*}{Theorem}
\newtheorem*{prop*}{Proposition}
\newtheorem*{cor*}{Corollary}
\newtheorem*{conj*}{Conjecture}
\numberwithin{equation}{section}
\theoremstyle{remark}
\newtheorem{rem}[lemma]{Remark}
\newtheorem*{rem*}{Remark}
\theoremstyle{definition}
\newtheorem{Def}[lemma]{Definition}
\newtheorem*{Def*}{Definition}

\begin{document}

\title{Lens rigidity for a particle in a Yang-Mills field}
\author[G. P. Paternain, G. Uhlmann, H. Zhou]{Gabriel P. Paternain, Gunther Uhlmann, Hanming Zhou}
\date{\today}
\address{Department of Pure Mathematics and Mathematical Statistics, University of Cambridge, Cambridge CB3 0WB, UK}
\email{g.p.paternain@dpmms.cam.ac.uk}
\address{Department of Mathematics, University of Washington, Seattle, WA 98195-4350, USA; HKUST Jockey Club Institute for Advanced Study, HKUST, Clear Water Bay, Kowloon, Hong Kong, China}
\email{gunther@math.washington.edu}
\address{Department of Mathematics, University of California Santa Barbara, Santa Barbara, CA 93106, USA}
\email{hzhou@math.ucsb.edu}

\begin{abstract} We consider the motion of a classical colored spinless particle under the influence of an external Yang-Mills potential $A$ on a compact manifold with boundary of dimension $\geq 3$.  We show that under suitable convexity assumptions, we can recover the potential $A$, up to gauge transformations, from the lens data of the system, namely, scattering data plus travel times between boundary points.

\end{abstract}

\maketitle

\section{Introduction} This paper considers a nonlinear geometric inverse problem associated with the motion of a classical colored spinless particle under the influence of an external Yang-Mills potential $A$. We shall show that under suitable conditions, it is possible to recover the potential $A$, up to gauge transformations, from the lens data of the system (i.e. scattering data plus travel times).

In order to set up the inverse problem, let us give first a brief description of the system in question and the physical background. Let $(M,g)$ be a compact connected Riemannian manifold with boundary and $G$ a compact Lie group of matrices with Lie algebra $\lg$. We think of $M$ as the configuration space where our classical colored particle travels and
we think of the dual $\lg^*$ as the space of ``color charges'' or internal degrees of freedom.
Since $G$ is compact we can fix once and for all a bi-invariant metric on $G$, or equivalently, we endow $\mathfrak{g}$ with an Ad-invariant metric $\langle\cdot,\cdot\rangle$; with this metric we identify $\mathfrak{g}^*$ with $\mathfrak{g}$.
In its most general form the motion takes place in the adjoint bundle of a principal bundle $P$ with structure group $G$.
For reasons of exposition we shall assume that the bundles are trivial; this is actually no serious restriction once we impose our global conditions on $M$, which reduces our global problem to a local one.
Thus we consider $P=M\times G$ and the adjoint bundle given $M\times \lg$. In this case, a connection $A$ (the external Yang-Mills potential) is just an element 
$A\in C^{\infty}(M,T^*M\otimes\mathfrak{g})=\Lambda^{1}(M,\mathfrak{g})$. Since $\mathfrak{g}$ is a Lie algebra of matrices, we can think of $A$ as a matrix of 1-forms in that Lie algebra. 
More generally, we can consider forms of any degree with values in $\mathfrak{g}$; the set of such forms with degree $k$
is denoted $\Lambda^{k}(M,\mathfrak{g})$. 
Given an external potential $A\in \Lambda^{1}(M,\mathfrak{g})$, we can associate to it a fundamental quantity: its {\it curvature} or field strentgh. It is defined as
\[F:=F_{A}=dA+A\wedge A\in \Lambda^{2}(M,\mathfrak{g})\]
where $(A\wedge A)_{x}(v,w):=[A_{x}(v),A_{x}(w)]$ for $x\in M$ and $v,w\in T_{x}M$, and $[\cdot,\cdot]$ is the commutator of matrices.
Using the metric $g$, given $\xi\in\lg$, we can define a $(1,1)$-tensor $\mathbb F^{\xi}:TM\to TM$ uniquely by
\[g_{x}(\mathbb{F}^{\xi}_{x}(v),w)=\langle F_{x}(v,w),\xi\rangle\]
for all $x\in M$ and $v,w\in T_{x}M$. The field $\mathbb F$ will play the role of a generalized Lorentz force.
The connection $A$ induces a covariant derivative in the adjoint bundle which we denote by $D$.

We are now ready to write down the equations of motion. The system lives in $TM\times \lg$ and the ODEs determining the trajectories $t\mapsto (\gamma(t),\dgamma(t),\xi(t))\in TM\times\lg$ are given by

\begin{equation}\label{Wong's}
\left\{ \begin{array}{lr}
 \nabla_{\dot\gamma}\dot\gamma=\mathbb{F}^{\xi}_{\gamma}(\dgamma), \\[.5em]
   D_{\dot\gamma}\xi=0,
       \end{array} \right.
\end{equation}
where $\nabla$ is the Levi-Civita connection of $g$. These coupled equations are referred to as {\it Wong's equations}. Wong introduced them as a model for the motion for a classical colored spinless particle under the influence of an external Yang-Mills potential $A$ \cite{W70}. The first equation in \eqref{Wong's} describes a particle under the influence of a generalized Lorentz force parameterized by the color charge $\xi$ while the second asserts that the color charge is parallel translated along this trajectory in configuration space. 
The equations reduce to the Lorentz equations in the abelian case $G=U(1)$. The Kaluza-Klein framework gives an alternative description in which the particle travels in a geodesic relative to a certain metric on the principal bundle $P$. Kerner \cite{Ke68} generalized Kaluza-Klein's idea from the abelian to the non-abelian case. The Wong and the Kaluza-Klein formulations were put into a symplectic framework by Sternberg \cite{Ste77} and Weinstein \cite{We78} respectively. See \cite{Sni79,M84} for a further discussion of symplectic aspects; the books \cite{GS90,Mo2002} contain extensive background on these equations.
It is worth noting that while Wong's equations might not have effective physical consequences to high energy physics (there is no such a thing as a ``classical quark''), the equations do appear in many different contexts. For instance they appear naturally when considering the so-called {\it isoholonomic problem}: fix $x_0,x_{1}\in M$; among all curves joining $x_0$ to $x_1$ with fixed parallel transport operator between $x_0$ and $x_1$ (with respect to $A$), find the one with smallest length.
Montgomery proved in \cite{Mo90} that $\gamma$ is extremal for the isoholonomic problem iff there exists $\xi(t)$ such that $(\gamma(t),\xi(t))$ solves Wong's equations. For other fascinating connections (like the Cat's problem and non-abelian Berry's phase), we refer to \cite{Mo90,Mo2002}. Equations \eqref{Wong's} also appear as suitable semi-classical limits of connection Laplacians \cite{HPS83,ST89}.

A quick analysis of \eqref{Wong's} reveals two kinematic constraints: $\gamma$ must travel at constant speed and $\xi$ must remain in the adjoint orbit it started on (the latter constraint implies in particular that the norm $\|\xi\|$ of the color charge is always a constant of the motion). For this reason, it makes sense from now on to restrict our motion to the compact phase space $SM\times \mathcal O$, where $SM$ is the unit sphere bundle of $M$ and $\mathcal O$ is a fixed adjoint orbit in $\lg$.  This defines a flow $\Phi_{t}:SM\times\mathcal O\to SM\times\mathcal O$ and we shall use $\phi=(x,v,\xi)\in SM\times \mathcal O$ to denote points in the phase space. We shall refer to the curve $\gamma$ as a {\it Yang-Mills geodesic} (YM-geodesic). 

For any $\phi=(x,v,\xi)\in SM\times \mathcal O$, we define the travel time 
$$\tau: SM\times\mathcal O\to [0,+\infty]$$
to be the first non-negative time when the YM-geodesic $\gamma$ determined by $\phi\in SM\times\mathcal O$ exits $M$. If $\tau$ is finite for any $\phi\in SM\times \mathcal O$, then we say that the pair $(g,A)$ (or the flow $\Phi_{t}$) is {\it non-trapping}. Define
$$\partial_\pm SM = \{(x,v) \in SM \,;\, x \in \partial M, \pm\langle v,\nu \rangle \leq 0 \}$$
where $\nu$ is the outer unit normal vector to $\partial M$. We denote the restriction of $\tau$ onto $\p_+SM\times \mathcal O$ by $\ell$, i.e. $\ell:=\tau|_{\p_+SM\times \mathcal O}$.
The {\it scattering relation} is the map
$$\mathcal S:\p_+SM\times \mathcal O \to \p_-SM\times \mathcal O$$
given by $\mathcal S(\phi):=\Phi_{\ell(\phi)}(\phi)$.
The data $(\mathcal S, \ell)$ constitute the {\it lens data} for the system and our inverse problem is to recover the external potential $A$ from the lens data $(\mathcal S,\ell)$. As it is common with problems involving connections, the problem has a natural ambiguity. Given a smooth map $u:M\to G$ (a gauge), it is well known that 
\[\tilde{A}:=u^{-1}du+u^{-1}Au\]
is a connection with curvature
\[F_{\tilde{A}}=u^{-1}F_{A}u.\]
This implies $\mathbb{F}_{\tilde{A}}^{\xi}=\mathbb{F}_{A}^{u\xi u^{-1}}$ since the inner product in the Lie algebra is invariant under the adjoint action. Using this, one may easily check that if $(\gamma(t), \xi(t))$ satisfies \eqref{Wong's} for $A$, then $(\gamma(t),u^{-1}(\gamma(t))\xi(t) u(\gamma(t)))$ satisfies \eqref{Wong's} for $\tilde{A}$.
If in addition we have $u|_{\partial M}=e$ (identity element in $G$), then $A$ and $\tilde{A}$ have the same lens data. Thus the inverse problem is:

\medskip

{\it Suppose $(\mathcal S_{A},\ell_{A})=(\mathcal S_{\tilde{A}},\ell_{\tilde{A}})$. Does there exist a smooth map $u:M\to G$ such that $u|_{\partial M}=e$ and $\tilde{A}=u^{-1}du+u^{-1}Au$?}

\medskip

We will solve this inverse problem under some assumptions, the most important of which is a convexity condition related to the flow $\Phi_{t}$. To describe this condition let us denote by $\mathbb{X}$ the vector field on $SM\times\mathcal O$ associated with $\Phi_t$. The flow is non-trapping iff there exists a smooth function in phase space  $h:SM\times\mathcal O\to\R$ such that $\mathbb{X}^2h>0$ \cite[Theorem 6.4.1]{DH72}. The global convexity condition is a considerable enhancement of non-trapping in which $h$ only depends on $x\in M$.
If $h(x,v,\xi)=f(x)$, then straightforward calculations using \eqref{Wong's} show that $\mathbb{X}h(x,v,\xi)=df_{x}(v)$ and
\[\mathbb{X}^2 h=\mathbb{X}(df)(x,v,\xi)=\text{Hess}_{x}(f)(v,v)+\langle F_{x}(v,\nabla f(x)),\xi\rangle.\]
Motivated by this, we give the following definition.

\begin{Def} A smooth function $f:M\to \R$ is said to be {\it strictly YM-convex} if
\[\text{Hess}_{x}(f)(v,v)+\langle F_{x}(v,\nabla f(x)),\xi\rangle>0\]
for all $(x,v,\xi)\in SM\times\mathcal O$. Similarly, we shall say that $x\in \partial M$ is strictly YM-convex if 
$$\Lambda_{x}(v,v)+\langle F_{x}(v,\nu(x)),\xi\rangle>0$$
for any $v\in S_x\p M$ and $\xi\in \mathcal O$, where $\Lambda$ is the second fundamental form of $\p M$.
If this holds for all $x\in\partial M$, then we say that $\partial M$ is strictly YM-convex.
\label{def:convex}
\end{Def}

We note that if $f$ is strictly YM-convex, then it is strictly convex in the usual sense in Riemannian geometry. Indeed, replacing $v$ by $-v$ in Definition \ref{def:convex} gives $\text{Hess}_{x}(f)(v,v)>0$
for all $(x,v)\in SM$.  This forces $M$ to be contractible (cf. \cite[Lemma 2.1]{PSUZ16}) and thus all bundles over $M$ are trivial; this explains why we considered trivial bundles from the start.
Observe also that the notion of strict YM-convexity depends on the adjoint orbit $\mathcal O$ that we have fixed. We are now ready to state our main global result.

\begin{thm}Let $(M,g)$ be a compact Riemannian manifold with boundary and dimension $\geq 3$ and let $\mathcal O$ be an adjoint orbit that contains a basis of $\lg$.
Let $A$ and $\tilde{A}$ be two Yang-Mills potentials such that
\begin{enumerate}
\item $\p M$ is strictly YM-convex with respect to both $(g,A)$ and $(g,\tilde{A})$;
\item $i^*A=i^*\tilde{A}$ where $i:\p M\to M$ is the canonical inclusion.
\end{enumerate}

If $(g,A)$ admits a strictly YM-convex function and $(\mathcal S_{A},\ell_{A})=(\mathcal S_{\tilde{A}},\ell_{\tilde{A}})$, then there exists a smooth function $u:M\to G$ such that $\tilde{A}=u^{-1}du+u^{-1}Au$ and $u|_{\p M}=e$.
 \label{thm:main}
\end{thm}

 
 Let us comment first on the condition on the adjoint orbit $\mathcal O$. Clearly some assumption is needed as the following trivial example shows: if $\mathcal O=\{0\}$ is the trivial adjoint orbit, then $\mathbb F^{0}=0$ and equations \eqref{Wong's} just become the equations of geodesics and $\xi(t)\equiv 0$. In this situation it is impossible to recover $A$ from lens data, since the latter does not take into account the external field $A$ at all. The condition that $\mathcal O$ contains a basis of $\lg$ ensures a proper coupling between $g$ and $A$ and it is actually easy to satisfy. For instance in the abelian case $G=U(1)$ one simply needs a non-zero charge: $\mathcal O=\{\xi\}$ with $0\neq \xi\in i\mathbb R$. In the case of $G=SU(2)$, the adjoint orbits are concentric spheres and we just need the sphere not to reduce to the origin. In fact for any simple Lie algebra $\lg$ the condition will hold as long as $\mathcal O$ is not trivial; this follows right away from the observation that the vector space spanned by $\mathcal O$ is an ideal in $\lg$. Similarly if $\lg$ semi-simple, the adjoint orbit of $\xi$ spans $\lg$ iff the projection of $\xi$ onto each simple factor of $\lg$ is non-zero.
In our proof of Theorem \ref{thm:main} the condition on $\mathcal O$ will naturally arise when proving ellipticity of certain pseudo-differential operators and when establishing a determination result for the boundary jet of the external field $A$.

 To illustrate why $i^*A=i^*\tilde{A}$ is required in Theorem \ref{thm:main}, let us consider again the abelian case $G=U(1)$. Since the adjoint orbits are just points, $\xi(t)$ is constant, so only the curvature of the external field is participating in Wong's equations \eqref{Wong's} and hence we cannot expect to recover more than the curvature of the external field. Theorem \ref{thm:main} shows that with additional boundary information like $i^*A=i^*\tilde{A}$, we can recover the external field up to gauge equivalence.
 
 The main assumptions in Theorem \ref{thm:main} are of course the convexity ones. We shall explain their relevance while we give a synopsis of the main ideas in the proof of the theorem. The proof follows the template laid out by Stefanov, Uhlmann and Vasy in their recent proof of boundary and lens rigidity for Riemannian metrics
 \cite{UV16,SUV16,SUV14,SUV17}.  We refer to \cite{Croke04,S08,U14,UZ16} for surveys, additional relevant references, and context for the lens and boundary rigidity problem in the Riemannian case.

The template when applied to our setting works as follows. (We emphasize that in our setting the metric $g$ is fixed and known and we are only interested in recovery of the external field $A$.) 
 \begin{enumerate}
 \item The main result underlying the proof is a local result near a convexity point of the boundary (see Theorem \ref{local thm} below) and the passage from local to global is achieved using the strictly YM-convex function by marching from the boundary into the inside in a layer stripping argument. Since our problem has a gauge this requires some care since some gluing and extensions of the local gauges are necessary. For this we use repeteadly the PDE satisfied by the gauges: $du=u\tilde{A}-Au$.

\item The local result is proved via a {\it pseudo-linearization} that reduces the nonlinear problem to a local linear one near a convexity point at the boundary.

\item The local linear problem involves of a new type of X-ray transform $I_{w}$ with weights. This transforms acts on pairs $[f,\beta]$, where $f$ is a $\lg$-valued 2-form and $\beta$ is a matrix valued 1-form.  The component $f$ is essentially the difference of the curvatures $F_{A}-F_{\tilde{A}}$.  We note that to be able to work with smooth weights we do need a boundary determination result for the jet of the connection in a suitable gauge.
Proving local injectivity of this transform will complete the proof.

\item To prove local injectivity we use the groundbreaking techniques from \cite{UV16}. We introduce a localized operator that plays the role of the normal operator $I^{*}_{w}I_{w}$ and we make sure it fits Melrose's scattering calculus \cite{Me94}. To obtain the Fredholm property in this calculus, one needs to prove that the boundary symbol is elliptic. We achieve this when we restrict the operator to (scattering) $\lg$-valued 2-forms.
Once the Fredholm property is derived, we prove injectivity when our connections are expressed in the 
 {\it normal gauge}.
   
\end{enumerate}

Theorem \ref{thm:main} illustrates how flexible and powerful the approach laid out by Stefanov, Uhlmann and Vasy is. In the non-abelian case, the system given by Wong's equations \eqref{Wong's} is unlike anything considered before since the motion has a component running in $\mathcal O$ affecting the curves in $M$. Implementing the scheme brings additional novel features, like the new X-ray transform mentioned in item (3) above. Unlike the boundary rigidity case \cite{SUV17} we will not need to make additional modifications to the operator to prove ellipticity thanks to the structure of \eqref{Wong's} that involve directly the curvature of the external field. 
 
A predecessor to Theorem \ref{thm:main} appears in \cite{HZ16} for the abelian case $G=U(1)$.  Another application of the scheme above are the results in \cite{PSUZ16} in which the problem of recovering a connection from parallel transport along geodesics is considered, but in this case, the underlying dynamical system (the geodesic flow) is unaffected by the external field.

Let us give now a large class of examples to which Theorem \ref{thm:main} applies. From the discussion above we know that we have to start with a Riemannian manifold $(M,g)$ with strictly convex boundary and admitting a strictly convex function $f$. Examples of such manifolds are discussed in detail in  \cite{PSUZ16} and include manifolds of non-negative sectional curvature (these could contain conjugate points).
Consider now two connections $A$ and $\tilde{A}$ which are compactly supported inside $M$; thus the boundary is strictly YM-convex with respect to both connections and $i^*A=i^*\tilde{A}$.
Select now an adjoint orbit $\mathcal O$ containing a basis of $\lg$ (we have already mentioned that if $G$ is simple, all non-trivial orbits will have that property). Given any positive number $\lambda$, the set $\lambda\mathcal O$ is obviously also an adjoint orbit containing a basis of $\lg$. If $\lambda$ is sufficiently small, we see from Definition \ref{def:convex} that a strictly convex function $f$ remains YM-convex and the theorem will apply.   On the other hand if $\lambda$ becomes too large, the effect of the Lorentz force increases and one may expect the dynamics to develop localized trapping (as in the abelian case), thus blocking YM-convexity.

This paper is organized as follows. Section \ref{sec_prelim} contains preliminaries and sets up the scene for the pseudo-linearization argument using the key integral identity from \cite{SU98}.
Section \ref{sec_jet} proves that after performing a suitable gauge transformation, the lens data determines the jet of the external field $A$ at the boundary. For this we need to assume that we know $i^*A$ and that $\mathcal O$ contains a basis of $\mathfrak{g}$. The argument is somewhat involved and uses the key identity \eqref{integral identity 2} below. The result on the jet is needed to ensure that later on we can modify the weights in the linear problem to make them smooth. Section \ref{section:local} states the local non-linear problem and explains how the pseudo-linearization is applied to reduce the non-linear problem to the injectivity of a linear X-ray transform with matrix weights. As already pointed out, this X-ray transform has not appeared in this form before in the literature and has some unique features produced by the specific form of the Wong equations \eqref{Wong's}.
In Section \ref{sec:injectivity of I_w} we prove injectivity for the local X-ray transform and Section \ref{section:proofglobal} completes the proof of the global Theorem \ref{thm:main}.

\bigskip

\noindent {\bf Acknowledgements.} 
GPP and HZ were supported by EPSRC grant EP/M023842/1. GU was partly supported by NSF, a Si-Yuan Professorship at HKUST and a FiDiPro Professorship of the Academy of Finland.
GPP and HZ would like to express their gratitude to the University of Washington at Seattle for hospitality while part of this work was being carried out. 
We are grateful to Leo Butler for a helpful discussion concerning adjoint orbits and to Alexander Strohmaier for pointing out that Wong's equations appear as semi-classical limits. Finally, we thank the referee for many valuable comments and corrections.

\section{Preliminaries} \label{sec_prelim}
In this section we collect various facts that will be needed later on. The first is the expression in local coordinates
of Wong's equations \eqref{Wong's}. We take coordinates $(z^{i},v^{i})$ in $TM$ and pick a basis $\{{\bf e}_{1},\cdots,{\bf e}_{d}\}$ of  $\mathfrak{g}$ so that $\xi_{\alpha}$ are the coordinates of $\xi\in\mathfrak g$ (later on we will assume that the basis is contained in a given adjoint orbit $\mathcal O$).
In these coordinates \eqref{Wong's} has the form
\begin{equation}\label{Wong's local}
\left\{ \begin{array}{lr}
 \frac{dz^i}{dt}=v^i, \\[.5em]
  \frac{dv^i}{dt}=-\Gamma_{jk}^iv^jv^k+ g^{ij} F^\alpha_{jk}v^k \xi_\alpha,\\[.5em]
  \frac{d\xi_\alpha}{dt}=-\xi_\beta c^\beta_{\alpha \mu}A^\mu_i v^i.
       \end{array} \right.
\end{equation}
Here $\Gamma_{jk}^i$'s are the Christoffel symbols with respect to the metric $g$, $A^{\alpha}_{i}$ are the components of the external field $A$ (connection), $F^\alpha_{jk}$'s are the components of the Yang-Mills field strength $F$ (curvature) and $c^\beta_{\alpha \mu}$'s are the structure constants of the Lie algebra in the chosen basis. In particular, the generating vector field of the flow $\Phi_{t}$ on $SM\times \mathcal O$ is 
\begin{equation}\label{generating vector}
\mathbb{X}=v^i\frac{\p}{\p z^i}+(-\Gamma_{jk}^iv^jv^k+g^{ij}\xi_\alpha F^\alpha_{jk}v^k)\frac{\p}{\p v^i}-\xi_\beta c^\beta_{\alpha \mu}A^\mu_i v^i\frac{\p}{\p \xi_\alpha}.
\end{equation}
We can of course also consider the flow $\Phi_t$ in $TM\times \mathfrak g$ with associated vector field $\mathbb{X}$ also given by \eqref{generating vector}. 

\subsection{Pseudo-linearization}

As we mentioned in the introduction the non-linear problem will be reduced to a linear one via a process we call {\it pseudo-linearization}. The idea behind it is quite simple and it is based on the following consideration first utilized in \cite{SU98}. Let $Z$ be a manifold
with two vectors fields $X_{i}$, $i=1,2$ and let $\Phi_{t}^{i}$ denote their flows.
Fix $x\in Z$ and $t>0$. Consider the curve
\[[0,t]\ni s\mapsto \Gamma(s):=\Phi^{2}_{t-s}\circ\Phi_{s}^{1}(x).\]
Obviously it connects the point $\Phi_{t}^{2}(x)$ to $\Phi_{t}^{1}(x)$. Computing the tangent vector to $\Gamma$ is straightforward using the chain rule:
\[\dot{\Gamma}(s)=d_{\Phi_{s}^{1}(x)}\Phi_{t-s}^{2}(X_{1}(\Phi_{s}^{1}(x))-X_{2}(\Phi_{s}^{1}(x))).\]
Thus
\[\int_{0}^{t}\dot{\Gamma}(s)\,ds=\Gamma(t)-\Gamma(0)=\Phi_{t}^{1}(x)-\Phi^{2}_{t}(x).\]
Of course in a manifold this does not make sense but it is certainly fine in Euclidean space.
Suppose now that $Z$ has boundary and given $x\in Z$, let $\tau_i(x)$ be the exit time for $s\mapsto \Phi^{i}_{s}(x)$, $i=1,2$. Thus if we set $t=\tau_1(x)$ we see that $t-s=\tau_{1}(\Phi_{s}^{1}(x))$.
Hence if we introduce the {\it weight}
\[\mathcal W(x):=d_{x}\Phi^{2}_{\tau_{1}(x)}\]
the integral above becomes
\begin{equation}
\int_{0}^{\tau_{1}(x)}\mathcal W(X_{1}-X_{2})(\Phi^{1}_{t}(x))\,dt=\Phi^{1}_{\tau_{1}(x)}(x)-\Phi^{2}_{\tau_{1}(x)}(x).
\label{eq:psl}
\end{equation}
Having the same lens data means that $\tau_{1}(x)=\tau_{2}(x)$ and $\Phi^{1}_{\tau_{1}(x)}(x)=\Phi^{2}_{\tau_{2}(x)}(x)$ whenever $x\in \p Z$ (assuming these times are finite) and hence the right hand side in \eqref{eq:psl} vanishes.

The idea is to apply this to $Z=TM\times\lg$ and the flows $\Phi_{t}$ and $\tilde{\Phi}_{t}$ corresponding to two different external fields $A$ and $\tilde{A}$. Hence we take coordinates in $M$ that naturally give us coordinates in $TM$. We also pick a basis of the Lie algebra $\lg$ as before and we work as if we were in Euclidean space. Since we will apply \eqref{eq:psl} locally this is no restriction at all.


Let $\Phi(t,\phi):=\Phi_{t}(\phi)$,  $\dim M=n$, $\dim G=d$ and write in local coordinates $\Phi:=(X,\Theta,\Xi)$. Hence if we fix $t$,
\begin{equation}\label{weight matrix}
\frac{\p \Phi}{\p \phi}(t,\phi)=\begin{pmatrix} \frac{\p X}{\p z} & \frac{\p X}{\p v} & \frac{\p X}{\p \xi}\\[.5em]  \frac{\p \Theta}{\p z} & \frac{\p \Theta}{\p v} & \frac{\p \Theta}{\p \xi}\\[.5em] \frac{\p \Xi}{\p z} & \frac{\p \Xi}{\p v} & \frac{\p \Xi}{\p \xi} \end{pmatrix}
\end{equation}
is a $(2n+d)\times (2n+d)$ matrix function on $TM\times \mathfrak g$. Since $\Phi(0,\phi)=\phi$, it is clear that for all $\phi=(z,v,\xi)\in T M\times \mathfrak g$
\begin{equation}\label{t=0}
\frac{\p \Phi}{\p \phi}(0,\phi)=\text{Id}_{(2n+d)\times (2n+d)}.
\end{equation}
The weight $\mathcal W$ is just
\[\mathcal W(\phi)=\frac{\p \tilde{\Phi}}{\p \phi}(\tau(\phi),\phi).\]
and hence \eqref{eq:psl} becomes the following very useful identity which we note here and will be used repeatedly later on:
\begin{equation}
\int_0^{\tau(\phi)}\frac{\p \tilde{\Phi}}{\p \phi}(\tau(\phi)-s,\Phi(s,\phi))(\mathbb X-\tilde{\mathbb X})(\Phi(s,\phi))\,ds=\Phi(\tau(\phi),\phi)-\tilde\Phi(\tau(\phi),\phi).
\label{integral identity 2}
\end{equation}
In particular, we will restrict $\phi$ to the compact phase space $SM\times \mathcal O$.

\section{Determination of the boundary jet}
\label{sec_jet}

As a first step towards the proof of Theorem \ref{thm:main}, we will make use of \eqref{integral identity 2} to show that the lens data determines the boundary jet of $A$, therefore $F$, up to a gauge transformation (note that in this paper the metric $g$ is given).

We consider boundary normal coordinates $z=(z',z^n)$, $z'=(z^1,\cdots,z^{n-1})$, with respect to $g$ on some neighborhood $U$ near $p$, so $z(p)=0$, $g=g_{ij}dz^i dz^j+(dz^n)^2$, $i,j<n$. Then we write $A=A_i dz^i$ with $A_i$, $i=1,\cdots,n$, (locally) Lie algebra valued functions. The curves (geodesics) $\gamma(t)=\{(z',t): z'\,\, \mbox{fixed}\}$ are normal to the boundary $z^n=0$ with $\dot\gamma=\p_n$, and $\gamma$ depends on $z'$ smoothly.  Let $u:[0,\varepsilon) \to G$ solve the following transport equation along each $\gamma$
$$\dot u+A_{\gamma}(\dot\gamma)u=0,\quad u(0)=e.$$
Then $u$ induces a smooth map $u:U\to G$ and
$$\p_n u(z)+A_n(z)u(z)=0, \quad u|_{z^n=0}=e.$$ 
The advantage of the boundary normal coordinates in the construction of the map $u$ is that it easily induces a global gauge transformation, such that in some collar neighborhood of $\p M$, $u$ is defined by the above construction.
If we set $A'=u^{-1}d u+u^{-1} A u$ (notice that $(g,A)$ and $(g,A')$ have the same lens data), by a simple calculation one can show that $A'_n=0$ in $U$ near $p\in \p M$. In addition, $\iota^* A=\iota^* A'$ where $\iota:\p M\to M$ is the inclusion map. To determine the boundary jet of $A'$ near $p$, we only need to determine $\p^k_n A'_j|_{z^n=0}$ for all $k\geq 0$ and $j<n$. Observe that $F'=u^{-1}Fu$, $u|_{\p M}=e$, thus $p\in \p M$ is strictly YM-convex with respect to $(g,A)$ if and only if it is strictly YM-convex with respect to $(g,A')$.


\begin{prop}\label{boundary jet of A} Assume that $\p M$ is strictly YM-convex at $p\in \p M$ with respect to $(g,A)$ and that $\mathcal O$ contains a basis of $\mathfrak{g}$. There exists $u:M\to G$ with $u=e$ on $\p M$ such that if we define $A'=u^{-1}du+u^{-1}Au$, then the lens data $(\mathcal S,\ell)$ on $(z,v,\xi)$ for $(z,v)\in \p_+SM$ close to $S_p \p M$ and any $\xi\in \mathcal O$, together with $\iota^* A$ near $p$, determine $\p^k_n A'_j(p)$ for any $k\geq 0$, $j<n$. 
\end{prop}

The result of Proposition \ref{boundary jet of A} is local; indeed to determine the boundary jet of $A'$ at $p$, we only require $u$ to be defined near $p$ in $M$. 

\begin{proof}


As discussed above, after some gauge transformation, we may assume that the connection $A$ has normal component $A_{n}=0$. To use the integral identity \eqref{integral identity 2}, let $A_0$ just be the zero connection, so $F_0=0$ too. 
Let $\X_0$ be the corresponding generating vector for $(g,A_0)=(g,0)$, $\phi\in \p_+SM\times \mathcal O$, then
\begin{equation}\label{jet 1}
\int_0^{\ell(\phi)}\frac{\p {\Phi_0}}{\p \phi}(\ell(\phi)-s,\Phi(s,\phi))(\X-\X_0)(\Phi(s,\phi))\,ds=\Phi(\ell(\phi),\phi)-\Phi_0(\ell(\phi),\phi).
\end{equation}
The right-hand side of \eqref{jet 1} is totally determined by the lens data and the metric $g$.
 
 
Now given $w\in S_p\p M$, we define a smooth map $v:[0,\delta)\to S_p M$ for $0<\delta\ll 1$ such that in the boundary normal coordinates
$$v(t)=v'(t)\frac{\p}{\p z'}+v^n(t)\frac{\p}{\p z^n}=a(t)w+bt\frac{\p}{\p z^n},$$
where $b$ is a positive constant. Since $v(t)\in S_p M$, $a(t)=\sqrt{1-(bt)^2}$ and $v(0)=w$.
Since $\p M$ is strictly YM-convex at $p$, given $\xi\in \mathcal O$ and $t\in (0,\delta)$, there is a unique YM-geodesic $\gamma_{(p,v(t),\xi)}(s)$ with $\gamma_{(p,v(t),\xi)}(0)=p$, $\dot\gamma_{(p,v(t),\xi)}(0)=v(t)$, which exits $M$ at $\gamma_{(p,v(t),\xi)}(\ell(p,v(t),\xi))  \in \p M$ close to $p$. 
For fixed $\xi$, denote $(p,v(t),\xi):=\phi(t)$, $\ell(p,v(t),\xi):=\ell(t)$. Since $p\in\p M$ is strictly YM-convex, using the implicit function theorem, one can check that $\ell(t)$ is smooth on $[0,\delta)$. Moreover, $\ell(t)$ is monotonically decreasing as $t\to 0$. We claim that $\ell'(0)\neq 0$.

To prove the claim, we write $\ell(t)$ asymptotically
\begin{equation*}
\ell(t)=\ell(0)+\ell'(0) t+O(t^2)= \ell'(0)t +O(t^2).
\end{equation*}
If $\ell'(0)=0$, then the decay rate of $\ell(t)$, as $t\to 0$, is at least quadratic.
On the other hand, in the boundary normal coordinates if we write $\gamma_{\phi(t)}(s)=(z'_t(s),z^n_t(s))$, then
\begin{equation*}
z^n_t(s)=z^n_t(0)+\dot z^n_t(0) s+O(s^2)=v^n(t) s+O(s^2)=bt s+O(s^2). 
\end{equation*}
By assumption, $z^n_t(\ell(t))=0$, then 
\begin{equation}\label{if ell'(0)=0}
0=bt\, \ell(t)+O(\ell^2(t))
\end{equation}
for any $t\in [0,\delta)$. Since $b> 0$ and $\ell(t)$ decays faster than any linear expression as $t\to 0$, the first term on the right hand side of \eqref{if ell'(0)=0} decays strictly slower than the second term as $t\to 0$. This implies that (for sufficiently small $\delta>0$) \eqref{if ell'(0)=0} cannot hold on the whole interval $[0,\delta)$, which produces a contradiction, therefore $\ell'(0)\neq 0$. Moreover, since $\ell(0)=0$ and $\ell(t)>0$ for $t>0$, we must have $\ell'(0)>0$. By rescaling the parameter $b$ of $v(t)$, we may assume that $\ell'(0)=1$. Note also that $\p_t v^n(0)=b>0$.

Now, by \eqref{jet 1} 
\begin{equation}\label{jet 2}
\begin{split}
\int_0^{\ell(t)} \frac{\p {\Phi_0}}{\p \phi}&(\ell(t)-s,\Phi(s,\phi(t)))(\X-\X_0)(\Phi(s,\phi(t)))\,ds=R(t).
\end{split}
\end{equation}
Note that we have full knowledge of $R(t)$. Denote $\Phi(t):=\Phi(\ell(t),\phi(t))$.
Taking the derivative of both sides of \eqref{jet 2} with respect to $t$, it follows that
\begin{equation}\label{jet 3}
\begin{split}
\frac{\p {\Phi_0}}{\p \phi}&(0,\Phi(t))(\X-\X_0)(\Phi(t))\,\ell'(t)\\
& +\int_0^{\ell(t)} \p_t\Big(\frac{\p {\Phi_0}}{\p \phi}(\ell(t)-s,\Phi(s,\phi(t)))(\X-\X_0)(\Phi(s,\phi(t)))\Big)\,ds=R'(t).
\end{split}
\end{equation}
Let $t\to 0$, so $\ell(t)\to 0$, $\ell'(t)\to 1$ and $\phi(t)\to (0,w,\xi)$, then \eqref{generating vector}, \eqref{t=0} and \eqref{jet 3} imply that
$$ Id_{(2n+d)\times (2n+d)}\begin{pmatrix} 0 \\[.5em] \xi_\alpha g^{ij}(0) F^\alpha_{jk}(0)  w^k\\[.5em] -\xi_\beta c^\beta_{\alpha\mu} A^\mu_k(0) w^k \end{pmatrix}=R'(0).$$
In particular, we get the values of 
$\xi_\alpha g^{ij}(0) F^\alpha_{jk}(0)  w^k$ from the equality above for any $\xi\in \mathcal O$ and $w\in S_p\p M$. 
Since $\mathcal O$ contains a basis of $\mathfrak g$, one obtains the values of $F_{ij}(0), \, 1\leq i\leq n,\, 1\leq j\leq n-1$. By definition $F=dA+A\wedge A$, since $A|_{\p M}$ is given by our assumption (so the tangential derivatives $\p_j A$, $j<n$ are automatically recovered from the boundary data), we get the values of $(dA)_{ij}(0)$, $1\leq i\leq n$, $1\leq j\leq n-1$, from $F_{ij}(0)$.
On the other hand, $\p_n A_j=(dA)_{nj}$ for $j<n$ (recall that $A_n=0$). Thus we uniquely determine $\p_n A_j(0)$ for $j<n$. Similarly we recover $\p_n A_j(z)$ for any $z\in \p M$ close enough to $p$.


Next we want to recover $\p_n^2 A_j(0)$. 
 Notice that the first term on the left hand side of \eqref{jet 3} is known now for $t$ sufficiently small, and hence we may rewrite the equality as
\begin{equation}\label{jet 4}
\int_0^{\ell(t)} \p_t\Big(\frac{\p {\Phi_0}}{\p \phi}(\ell(t)-s,\Phi(s,\phi(t)))(\X-\X_0)(\Phi(s,\phi(t)))\Big)\,ds=R(t)
\end{equation}
where $R(t)$ is a known quantity. From now on we always use $R(t)$ to denote quantities that are known for $t$ small. 
We use $\frac{\p \Phi_0}{\p \phi}(s,t)$ as the short notation for $\frac{\p \Phi_0}{\p \phi} (\ell(t)-s,\Phi(s,\phi(t)))$ and differentiate both sides of \eqref{jet 4} with respect to $t$ to derive
\begin{equation}\label{jet 5}
\begin{split}
\p_t \frac{\p \Phi_0}{\p \phi}(s,t)&\Big|_{s=\ell(t)}(\X-\X_0)( \Phi(t))\ell'(t)+\frac{\p (\X-\X_0)}{\p \Phi}(\Phi(t))\frac{\p \Phi(t)}{\p \phi}\phi'(t)\ell'(t)\\
& +\int_0^{\ell(t)} \p_t^2\Big(\frac{\p {\Phi_0}}{\p \phi}(\ell(t)-s,\Phi(s,\phi(t)))(\X-\X_0)(\Phi(s,\phi(t)))\Big)\,ds=R'(t),
\end{split}
\end{equation}
with 
\begin{equation*}
\phi'(t)=\frac{d}{dt} (0,v(t),\xi)=(0,v'(t),0).
\end{equation*}


Notice that the first term on the left hand side of \eqref{jet 5} is known for $t\geq 0$ small. Let $t\to 0$; using \eqref{t=0} we recover the value of
$$\frac{\p (\X-\X_0)}{\p \Phi}(\Phi(0))\frac{\p \Phi(0)}{\p \phi}\phi'(0)=\frac{\p(\X-\X_0)}{\p v}(0,w,\xi)v'(0).$$
However, the term on the right hand side above does not contain higher order derivatives of $A$, which means that taking the limit of \eqref{jet 5} as $t\to 0$ does not provide any new information.

To address the issue, we keep on taking derivatives with respect to $t$. As discussed above, we can rewrite \eqref{jet 5} as
\begin{equation}\label{jet 6}
\begin{split}
 &\frac{\p (\X-\X_0)}{\p \Phi}(\Phi(t))\frac{\p \Phi(t)}{\p \phi}\phi'(t)\ell'(t)\\
& +\int_0^{\ell(t)} \p_t^2\Big(\frac{\p {\Phi_0}}{\p \phi}(\ell(t)-s,\Phi(s,\phi(t)))(\X-\X_0)(\Phi(s,\phi(t)))\Big)\,ds=R(t).
\end{split}
\end{equation}
Now we differentiate both sides of \eqref{jet 6} with respect to $t$ to derive
\begin{equation}\label{jet 7}
\begin{split}
 & \frac{d}{dt}\Big(\frac{\p (\X-\X_0)}{\p \Phi}(\Phi(t))\frac{\p \Phi(t)}{\p \phi}\phi'(t)\ell'(t)\Big)+\p_t^2 \frac{\p \Phi_0}{\p \phi}(s,t)\Big|_{s=\ell(t)}(\X-\X_0)(\Phi(t))\ell'(t)\\
&+2 \p_t \frac{\p \Phi_0}{\p \phi}(s,t) \p_t(\X-\X_0)(\Phi(s,\phi(t)))\Big|_{s=\ell(t)}\ell'(t)+\p_t^2(\X-\X_0)(\Phi(s,\phi(t)))\Big|_{s=\ell(t)}\ell'(t)\\
& +\int_0^{\ell(t)} \p_t^3\Big(\frac{\p {\Phi_0}}{\p \phi}(\ell(t)-s,\Phi(s,\phi(t)))(\X-\X_0)(\Phi(s,\phi(t)))\Big)\,ds=R'(t).
\end{split}
\end{equation}
Let $t\to 0$; similar to the analysis after \eqref{jet 5}, the values of the second and third terms on the left hand side of \eqref{jet 7} are known at $t=0$. Thus we recover the value of the following limit
\begin{equation}\label{limit 1}
\lim_{t\to 0}\,\,\frac{d}{dt}\Big(\frac{\p (\X-\X_0)}{\p \Phi}(\Phi(t))\frac{\p \Phi(t)}{\p \phi}\phi'(t)\ell'(t)\Big)+\p_t^2(\X-\X_0)(\Phi(s,\phi(t)))\Big|_{s=\ell(t)}\ell'(t).
\end{equation}

Notice that asymptotically
$$\Phi(s,\phi(t))=\phi(t)+s\X(\phi(t))+O(s^2),$$
thus 
\begin{equation*}
\p_t\Phi(s,\phi(t))=\phi'(t)+O(s),\quad \p_t^2\Phi(s,\phi(t))=\phi''(t)+O(s),
\end{equation*}
and we get that
\begin{equation*}
\lim_{t\to 0}\,\,\p_t^2(\X-\X_0)(\Phi(s,\phi(t)))\Big|_{s=\ell(t)}
=\frac{\p^2(\X-\X_0)}{\p \Phi^2}(\phi(0))(\phi'(0))^2+\frac{\p (\X-\X_0)}{\p \Phi}(\phi(0))\phi''(0).
\end{equation*}
Using that $\phi(t)=(0,v(t),\xi)$, a simple calculation shows that the second term in \eqref{limit 1} does not contain the normal derivatives of $F$ (the curvature), i.e. it is known already. On the other hand, for $\Phi(t)=\Phi(\ell(t),\phi(t))$ we have
\begin{equation*}
\begin{split}
\frac{d}{dt}\Phi(t)&=\phi'(t)+\ell'(t)\X(\phi(t))+O(\ell(t));\\
\frac{\p \Phi(t)}{\p \phi}&=Id+\ell(t)\frac{\p \X}{\p \Phi}(\phi(t))+O(\ell^2(t)),\quad \frac{d}{dt}\frac{\p \Phi(t)}{\p \phi}=\ell'(t)\frac{\p \X}{\p \Phi}(\phi(t))+O(\ell(t)).
\end{split}
\end{equation*}
Thus
\begin{equation*}
\begin{split}
&\lim_{t\to 0}\,\,\frac{d}{dt}\Big(\frac{\p (\X-\X_0)}{\p \Phi}(\Phi(t))\frac{\p \Phi(t)}{\p \phi}\phi'(t)\ell'(t)\Big)\\
=& \frac{\p^2(\X-\X_0)}{\p \Phi^2}(\phi(0))\Big(\phi'(0)+\X(\phi(0))\Big)\phi'(0)+\frac{\p (\X-\X_0)}{\p \Phi}(\phi(0))\frac{\p \X}{\p \Phi}(\phi(0))\phi'(0)\\
& + \frac{\p (\X-\X_0)}{\p \Phi}(\phi(0))\Big(\phi''(0)+\phi'(0)\ell''(0)\Big).
\end{split}
\end{equation*}
Based on \eqref{generating vector}, the first component of $\X(\phi(0))$ is $w$ which is tangential, then one can easily check that the first and third terms above are known, i.e. they do not contain the normal derivatives of $F$.

The analysis above implies that we can recover the value of
\begin{equation}\label{limit 2}
\frac{\p (\X-\X_0)}{\p \Phi}(\phi(0))\frac{\p \X}{\p \Phi}(\phi(0))\phi'(0)=\frac{\p (\X-\X_0)}{\p \Phi}(0,w,\xi)\frac{\p \X}{\p v}(0,w,\xi)v'(0).
\end{equation} 
By \eqref{generating vector}, in matrix form,
$$\frac{\p \X}{\p v}(0,w,\xi)=\begin{pmatrix}Id_{n\times n} \\ U_v \\ U_{\xi}\end{pmatrix},$$
where the matrices $U_v$ and $U_{\xi}$ are known. Therefore the term containing the derivatives of $F$ in \eqref{limit 2} has the form
\begin{equation*}
\xi_\alpha g^{ij}(0)\p_m F^\alpha_{jk}(0)w^k \p_t v^m(0),
\end{equation*}
which is the new information we obtain from \eqref{limit 2}.
As before, since $\mathcal O$ contains a basis of $\mathfrak g$ (recall that we have recovered $F$ and we also know the tangential derivatives $\p_m F_{jk}$, $m<n$) we can determine the values of 
$$\p_n F_{nk}(0)\p_t v^n(0),\quad k<n.$$
Since $\p_t v^n(0)\neq 0$, we get the values of $\p_n F_{nk}(0)$ for $k<n$. By the definition of $F$, and the fact that $A_n=0$, we obtain the values of $\p_n(d A)_{nk}(0)=\p_n^2 A_k(0)$ for $k<n$. We also recover $\p_n^2 A_k(z)$ for $z$ close to $p$ in a similar way.


Then we determine $\p_n^k A(0)$ for any $k>2$ by induction. Assume that we have recovered $\p_n^k A(0)$ for $k\leq K$ with some $K\geq 2$, and hence $\p_n^i F(0)$ for $i\leq K-1$. 
We differentiate \eqref{jet 4} $2K$ times with respect to $t$ to get
\begin{equation}\label{jet 8}
\begin{split}
\sum_{j=1}^{2K}&\frac{d^{2K-j}}{dt^{2K-j}}\Big(\p_t^j\big(\frac{\p {\Phi_0}}{\p \phi}(\ell(t)-s,\Phi(s,\phi(t)))(\X-\X_0)(\Phi(s,\phi(t)))\big)\Big|_{s=\ell(t)}\ell'(t)\Big)\\
&+ \int_0^{\ell(t)} \p_t^{2K+1}\Big(\frac{\p {\Phi_0}}{\p \phi}(\ell(t)-s,\Phi(s,\phi(t)))(\X-\X_0)(\Phi(s,\phi(t)))\Big)\,ds=\frac{d^{2K}}{dt^{2K}}R(t).
\end{split}
\end{equation}
As can be seen from the calculations for the case of $K=2$, the key information is encoded in the derivatives of $\X-\X_0$. Since we have recovered $\p_n^k F$ for $k\leq K-1$, we also know $\frac{\p^k (\X-\X_0)}{\p \Phi^k}(\Phi(t))$ for any $k\leq K-1$. On the other hand, for any $k\geq 0$
\begin{equation*}
\begin{split}
\p_t\frac{\p^k(\X-\X_0)}{\p \Phi^k}(\Phi(s,\phi(t)))&=\frac{\p^{k+1}(\X-\X_0)}{\p \Phi^{k+1}}(\Phi(s,\phi(t)))\Big(\phi'(t)+O(s)\Big),\\
\frac{d}{dt}\frac{\p^k(\X-\X_0)}{\p \Phi^k}(\Phi(t))&=\frac{\p^{k+1}(\X-\X_0)}{\p \Phi^{k+1}}(\Phi(t))\Big(\phi'(t)+\ell'(t)\X(\phi(t))+O(\ell(t))\Big).
\end{split}
\end{equation*}
Recall that either $\phi'(0)$ or $\phi'(0)+\X(\phi(0))$ will eliminate the normal derivative of $\frac{\p^k(\X-\X_0)}{\p \Phi^k}(\phi(0))$, i.e. $\p_n \frac{\p^k(\X-\X_0)}{\p \Phi^k}(\phi(0))$. Then it is not difficult to check that the first part on the left hand side of \eqref{jet 8} only contains $\p_n^k F$, $k\leq K$. Moreover $\p_n^K F$ only appears in the terms containing $\frac{\p^K (\X-\X_0)}{\p \Phi^K}(\Phi(t))$.

Recall again the argument of the case when $K=2$, to determine the value of $\p_n F(0)$ or equivalently $\frac{\p (\X-\X_0)}{\p \Phi}(\phi(0))$, we need the appearance of the term $\frac{\p \X}{\p \Phi}(\phi(t))\phi'(t)$. Asymptotically
\begin{equation*}
\begin{split}
\frac{d}{dt}\Big(\p_t \Phi(s,\phi(t))\Big|_{s=\ell(t)}\Big)&=\frac{d}{dt}\Big(\phi'(t)+\ell(t)\frac{\p \X}{\p \Phi}(\phi(t))\phi'(t)+O(\ell^2(t))\Big)\\
&=\phi''(t)+\ell'(t)\frac{\p \X}{\p \Phi}(\phi(t))\phi'(t)+O(\ell(t)).
\end{split}
\end{equation*}
For $\p_n^K F(0)$, we need $\frac{\p \X}{\p \Phi}(\phi(t))\phi'(t)$ to appear $K$ times. 

Based on the above analysis, one can show that \eqref{jet 8} has the following form, with some non-zero constant $C_0$,
\begin{equation*}
\begin{split}
C_0\frac{\p {\Phi_0}}{\p \phi}(0,\Phi(t))& \frac{\p^K (\X-\X_0)}{\p \Phi^K}(\Phi(t))\Big[\frac{d}{dt}\Big(\p_t \Phi(s,\phi(t))\Big|_{s=\ell(t)}\Big)\Big]^K \ell'(t)+G(t)\\
&+ \int_0^{\ell(t)} \p_t^{2K+1}\Big(\frac{\p {\Phi_0}}{\p \phi}(\ell(t)-s,\Phi(s,\phi(t)))(\X-\X_0)(\Phi(s,\phi(t)))\Big)\,ds=R(t),
\end{split}
\end{equation*}
where $G(t)$ is known at $t=0$. Now let $t\to 0$, we recover the value of
\begin{equation*}
\frac{\p^K (\X-\X_0)}{\p \Phi^K}(\Phi(0))\Big(\frac{\p \X}{\p \Phi}(\phi(0))\phi'(0)\Big)^K.
\end{equation*} 
In local coordinates, by the assumption that $\mathcal O$ contains a basis of $\mathfrak g$ we get the values of
\begin{equation*}
\p_n^K F_{nk}(0)\big(\p_t v^n(0)\big)^K,\quad  k<n.
\end{equation*}
Since $\p_t v^n(0)\neq 0$, we recover $\p_n^K F_{nk}(0)$, and consequently $\p_n^{K+1} A_k(0)$ for $k<n$.

This completes the proof.
\end{proof}

\begin{rem}
As already noted, in this paper we have assumed that the metric $g$ is given. However, we can recover the boundary jets of both the metric $g$ and Yang-Mills potential $A$ from the lens data by the same method as in the proof of Proposition \ref{boundary jet of A}. For instance, we can take the auxiliary system to be $(g_0,0)$, where $g_0$ is a constant metric in the boundary normal coordinates. Notice that in \eqref{generating vector}, the Christoffel symbols $\{\Gamma^i_{jk}\}$ produce a quadratic dependence in $v$, while the terms involving $F$ are linear in $v$, so we can separate the information regarding $g$ and $A$ respectively. The determination of the boundary jet of a Riemannian metric $g$ from its lens data (or boundary distance function) has been considered before in \cite{LSU03,UW03,SU09}. The reference 
\cite{UW03} employs the integral identity \eqref{integral identity 2}, while \cite{SU09} considers the case where the boundary is not necessarily convex.
 A boundary determination problem related to polarization tomography is studied in \cite{Ho09}, where the metric $g$ is known and the underlying dynamical system is the usual geodesic flow. 

Notice that the proof of Proposition \ref{boundary jet of A} is constructive. In particular, when $G=U(1)$ our method gives a construction of the boundary jet of a magnetic field from its corresponding lens data.
\end{rem}

Proposition \ref{boundary jet of A} easily implies the following uniqueness result.

\begin{lemma}\label{same boundary jet} Suppose that $\mathcal O$ contains a basis of $\mathfrak{g}$.
Consider two Yang-Mills potentials $A$ and $\tilde A$ such that $\p M$ is strictly YM-convex with respect to both $(g,A)$ and $(g,\tilde{A})$, they have the same lens data, and $\iota^*A=\iota^*\tilde A$. There exists $u:M\to G$, $u|_{\p M}=e$, such that if we define $A'=u^{-1}du+u^{-1}Au$, then $A'$ and $\tilde A$ have the same boundary jet. 
\end{lemma}

 
\begin{rem}\label{WLOG}
Without loss of generality, from now on we only need to consider the lens rigidity problem for potentials $A$, $\tilde A$ with the same boundary jet. 
\end{rem}


\section{The local problem and pseudo-linearization}
\label{section:local}

In this section we set up the local problem. Assume that $\p M$ is strictly YM-convex at $p\in \p M$. 
For $(z,v)\in \p_+SM$ close to $S_p\p M$ and any $\xi\in \mathcal O$, the YM-geodesic $\gamma_\phi$, $\phi=(z,v,\xi)$, will stay in some small neighborhood of $p$ in $M$ and exits the neighborhood from $\p M$ again by the convexity assumption. The main local result is the following theorem which is of independent interest.

\begin{thm}\label{local thm}

Let $(M,g)$ be a compact Riemannian manifold with boundary and dimension $\geq 3$ and let $\mathcal O$ be an adjoint orbit that contains a basis of $\lg$.
Let $A$ and $\tilde{A}$ be two Yang-Mills potentials such that
\begin{enumerate}
\item $\p M$ is strictly YM-convex at $p\in \p M$ with respect to both $(g,A)$ and $(g,\tilde{A})$;
\item $i^*A=i^*\tilde{A}$ on $\p M$ near $p$.
\end{enumerate}

If $(\mathcal S,\ell)=(\tilde{\mathcal S},\tilde{\ell})$ for $(z,v,\xi)\in \p_+SM\times \mathcal O$ with $(z,v)$ near $S_p \p M$, then there exists a smooth function $u:M\to G$ with $u(z)=e$ for $z\in \p M$ close to $p$, such that $\tilde{A}=u^{-1}du+u^{-1}Au$ in $M$ near $p$.
\end{thm}

We will now use \eqref{integral identity 2} and the result on the boundary jet given by Lemma \ref{same boundary jet} to reduce the local non-linear problem to a linear integral geometry problem, i.e. a suitable X-ray transform with weights. The key identity \eqref{integral identity 2} gives us right away:

\begin{lemma}\label{integral identity lemma}
Assume that $\mathcal S(\phi)=\tilde{\mathcal S}(\phi)$ and $\ell(\phi)=\tilde\ell (\phi)$ for some $\phi\in \p_+SM\times \mathcal O$, then
\begin{equation}\label{integral identity 3}
\int_0^{\ell(\phi)}\frac{\p \tilde{\Phi}}{\p \phi}(\ell(\phi)-s,\Phi(s,\phi))(\mathbb X-\tilde{\mathbb X})(\Phi(s,\phi))\,ds=0.
\end{equation}
\end{lemma}

By \eqref{generating vector}
$$\mathbb X-\tilde{\mathbb X}=\xi_\alpha g^{ij}(F^{\alpha}_{jk}-\tilde F^{\alpha}_{jk})v^k\frac{\p}{\p v^i} -\xi_\alpha c^\alpha_{\beta \mu}(A^\mu_{k}-\tilde A^\mu_{k})v^k\frac{\p}{\p \xi_\beta}.$$ 
Thus from \eqref{integral identity 3}, if we denote for short the structure constants by $\hat c$
\begin{equation}\label{integral identity 4}
\int_0^{\ell(\phi)}\frac{\p \tilde{\Phi}}{\p \phi}(\ell(\phi)-s,\Phi(s,\phi))\begin{pmatrix} 0 \\  g^{-1}  \mathcal F(v)\cdot \xi \\ - \hat c \mathcal A(v)\cdot \xi \end{pmatrix}(\Phi(s,\phi))\,ds=0,
\end{equation}
where $\mathcal F=F-\tilde F$, $\mathcal A=A-\tilde A$ and $\cdot$ denotes the inner product in the Lie algebra.  Note that generally $\mathcal F$ is not the curvature of the new Yang-Mills potential $\mathcal A$. 
We only take the second row of \eqref{weight matrix}, namely
\begin{equation}\label{integral identity 5}
\int_0^{\ell(\phi)}\bigg(\frac{\p \tilde \Theta}{\p v} g^{-1}  \mathcal F(v)\cdot \xi -\frac{\p \tilde \Theta}{\p \xi}  \hat c  \mathcal A(v)\cdot \xi\bigg)(\ell(\phi)-s,\Phi(s,\phi))\,ds=0.
\end{equation}
Let $\Phi(s,\phi)=(\gamma(s),\dot\gamma(s),\xi(s))$ and $\tau(z,v,\xi)$ be the exit time for $(z,v,\xi)\in SM\times \mathcal O$, as $\ell(\phi)-s=\tau(\Phi(s,\phi))$, then \eqref{integral identity 5} can be rewritten as
\begin{equation}\label{integral identity 6}
\int \left\{W(\gamma(s),\dot\gamma(s),\xi(s))\mathcal F_{\gamma(s)}(\dot\gamma(s))\cdot \xi(s) + Q(\gamma(s),\dot\gamma(s),\xi(s))(\hat c\mathcal A)_{\gamma(s)}(\dot\gamma(s))\cdot \xi(s)\right\}\,ds=0.
\end{equation}
Here we treat $\F=(\F^\alpha_{jk}dx^k)_{n\times d}$ and $\hat{c}\A=(c^\alpha_{\beta \mu} \A^\mu_k dx^k)_{d\times d}$ locally as matrix valued 1-forms whose action on elements $\xi$ of the Lie algebra (as a column vector) is represented by vector inner product, and
\begin{align*}
W(\gamma(s),\dot\gamma(s),\xi(s))& :=\frac{\p \tilde \Theta}{\p v} \Big(\tau\big(\gamma(s),\dot\gamma(s),\xi(s)\big), \big(\gamma(s),\dot\gamma(s),\xi(s)\big)\Big) g^{-1}\big(\gamma(s)\big),\\
Q(\gamma(s),\dot\gamma(s),\xi(s))& :=-\frac{\p \tilde \Theta}{\p \xi}\Big(\tau\big(\gamma(s),\dot\gamma(s),\xi(s)\big), \big(\gamma(s),\dot\gamma(s),\xi(s)\big)\Big).
\end{align*}
Notice that $\p \tilde \Theta/\p v=Id_{n\times n}$ at $S_p\p M\times \mathcal O$, thus $W$ is invertible near $S_p\p M\times \mathcal O$. 
The left hand side of \eqref{integral identity 6} is a weighted X-ray transform of $[\mathcal F,\hat c\mathcal A]$ along a YM-geodesic $\gamma$ of $(g,A)$. 
We denote the left hand side of \eqref{integral identity 6} by $I_w[\mathcal F,\hat c\mathcal A]$. The linear inverse problem now is the local invertibility (up to natural gauge transformations) of $I_w$ near a strictly YM-convex point $p\in \p M$.

In the next section, we will use microlocal analysis to study this local invertibility question. To make the argument work, one necessary assumption is that the weights $W$ and $Q$ as functions on $SM\times \mathcal O$ are smooth near $S_p\p M\times \mathcal O$. However, generally this is not the case due to the lack of smoothness of the exit time $\tau(z,v,\xi)$ at $S(\p M)\times \mathcal O$. To remedy this inconvenience, we extend $M$ to a larger manifold $\tilde M$, and the metric $g$ smoothly to $\tilde M$. Moreover, by Lemma \ref{same boundary jet} and Remark \ref{WLOG}, we may extend $A$ and $\tilde A$ smoothly to $\tilde M$ with $A=\tilde A$ in $\tilde M\setminus M$, therefore $\supp [\F,\hat c\A]\subset M$.

Let $H$ be a (local) hypersurface near $p$ in $\tilde M$ such that $H$ is tangent to $\p M$ at $p$ and $H\cap M=\{p\}$. We can find  local coordinates $(x,y^1,\cdots y^{n-1})$ near $p$ in $\tilde M$ with $x$ the level set function and $y=(y^1,\cdots, y^{n-1})$ local coordinates on $H$ (e.g. by considering normal coordinates relative to $H$). In particular, we may assume that $H=\{x=0\}$ and $M\subset \{x\leq 0\}$. 
On the other hand, notice that $A$ and $\tilde A$ are extended identically into $\tilde M$. Since $\p M$ is strictly YM-convex at $p$ with respect to both $A$ and $\tilde A$, we may assume that $H$ is strictly YM-convex with respect to $A$ from $\{x\leq 0\}$. Then there exists $0<c_0\ll 1$ so that $\{x=\delta\}$ is strictly YM-convex with respect to $A$ from $\{x\leq \delta\}$ for any $|\delta|\leq c_0$. 

Given any $(z,v)\in \p_+SM$ close to $S_p \p M$ and $\xi\in \mathcal O$, we can extend the YM-geodesics $\gamma_{z,v,\xi}$ (with respect to $A$) and $\tilde\gamma_{z,v,\xi}$ (with respect to $\tilde A$) into $\tilde M$. Since $(g,A)$ and $(g,\tilde A)$ share the same lens data, by the convexity of the level sets of $x$, the two extended YM-geodesics are identical outside $M$ and exit $\{x\leq c_0\}$ in finite time. In other words, we can define the lens data (locally) on $\{x=c_0\}$. If $A$ and $\tilde A$ have the same lens data on $\p M$, they have the same lens data on $\{x=c_0\}$. Moreover, since $\{x=c_0\}$ and $\p M$ are disjoint, the new exit time function $\tau$ as defined on $S\{x\leq c_0\}\times \mathcal O$ is smooth in $S\{x<c_0\}\times \mathcal O$, which includes $SM\times \mathcal O$. Now we apply Lemma \ref{integral identity lemma} to $(z,v,\xi)\in \p_+S\{x\leq c_0\}\times \mathcal O$ and the resulting integral also equals zero due to equality of the lens data, and the corresponding new weights $W$ and $Q$ in \eqref{integral identity 6}, as defined on $S\{x\leq c_0\}\times\mathcal O$, are smooth near $S_p\p M\times\mathcal O$ as $\tau$. Notice that $[\F,\hat c\A]$ is supported in $M$, which means that we can modify the weights $W$ and $Q$ outside $M$ to make them smooth on the whole domain, without affecting the value of the integral. From now on, we only need to consider weighted ray transforms as \eqref{integral identity 6} with smooth weights $W$ and $Q$.



\section{Injectivity of the linear problem} \label{sec:injectivity of I_w}

\subsection{Preliminaries and strategy} Let us first describe the setting that we will use and the strategy of the proof. The latter is substantially based on the recent developments in \cite{UV16,SUV14}
and we will refer frequently to these references.

The problem is local, so the first step is to chose a convenient set of coordinates. We start with the coordinates
$(x,y^1,\cdots,y^n)$ near $p\in \p M$ from the previous section. To define the neighborhood near $p$ for our local problem, we pick some constant $c$ with $0<c<c_0$, and we may let $c$ be sufficiently close to $0$ later if necessary. 
Vectors which are close to $S_p\p M$ can be parameterized by $\lambda \p_x+\omega \p_y$ with $\omega\in \mathbb S^{n-2}$ and $\lambda\ll 1$. The convexity assumption of $\p M$ near $p$ implies that for any $(x,y)\in O:=\{x>-c\}\cap M$ and $\lambda$ sufficiently small (the upper bound depends on $(x,y)$), $\omega\in \mathbb S^{n-2}$, and any $\xi\in \mathcal O$, the curve $\gamma_{x,y,\lambda,\omega,\xi}$ will stay in the neighborhood $O$ before exiting from $\p M$. 

It is worth mentioning at this stage that the constant $c$ can be taken locally uniform, i.e. for $c$ sufficiently small, there exist $0<\epsilon\ll 1$, such that for $-\epsilon\leq \delta \leq \epsilon$, we may replace the set $O$ by $\{-c+\delta < x < \delta\}\cap M$ assuming that $c+\epsilon<c_0$ (so the level sets are strictly YM-convex) and carry out similar arguments in this translated neighborhood. This uniform property will be used in the proof of the global lens rigidity result later on.

For the sake of simplicity, we shift the level set function $x$ by $c$, so that $H=\{x=c\}$, thus the interior (or artificial) boundary of $O$ becomes $\{x=0\}$. Moreover, we denote $\Omega:=\{x>0\}\subset \tilde M$. See Figure 1 below. 

\begin{figure}[ht]\label{illustration}
\includegraphics[width=140mm]{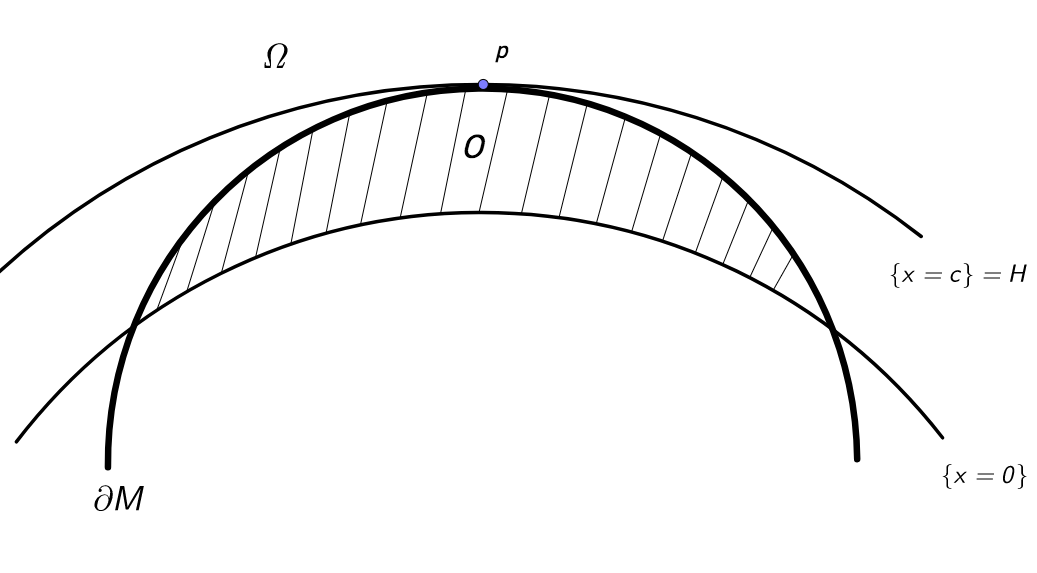}
\caption{The region below $\p M$ is the interior of $M$, while the region above $\{x=0\}$ is $\Omega$. The neighborhood $O=M\cap \Omega$ is the shadowed area. The hypersurface $H=\{x=c\}$ is tangent to $\p M$ at the point $p$.}
\end{figure}

Using the coordinates above, let 
$$\Phi(t)=(x(t),y(t),\lambda(t),\omega(t),\xi(t)),\quad \Phi(0)=(x,y,\lambda,\omega,\xi).$$
Given a $(n\times d)$-matrix valued 1-form $f$ and $(d\times d)$-matrix valued 1-form $\beta$, and motivated by the analysis in the previous section, we introduce a ray transform $I_w$ of the form
\begin{align*}
I_w[f,\beta](x,y,& \lambda,\omega,\xi):=\\
& \int \Big(W(\Phi(t)) f(x(t),y(t),\lambda(t),\omega(t))+Q(\Phi(t))\beta(x(t),y(t),\lambda(t),\omega(t))\Big)\cdot \xi(t)\,dt.
\end{align*}
To be consistent with the non-linear problem, we assume that the pair $[f,\beta]$ is supported in $M$, so the integral can be taken over $\mathbb R$. Recall that we are allowed to pick a basis of the Lie algebra $\mathfrak g$, denoted by $\{{\bf e}_1,\cdots, {\bf e}_d\}$, contained in the adjoint orbit $\mathcal O$. Without loss of generality, we may assume that under proper local coordinates ${\bf e}_j$ is the column vector whose entries are zeros except the $j$-th entry, which is $1$.
It is convenient for the time being to think of $f$ and $\beta$ as unrelated; later on they will be coupled as $[\mathcal F,\hat{c}\mathcal A]$ in the previous section.

Our goal is to understand the injectivity properties of $I_{w}$ near $p$. The main idea in \cite{UV16} is to consider a localized version of the normal operator $I_{w}^{*}I_{w}$ so that after careful adjustments it belongs to Melrose's scattering calculus with the objective of proving ellipticity in this calculus. This requires a careful analysis of the 
operator near the artificial boundary $\{x=0\}$, but provided that ellipticity is achieved, then the Fredholm property follows and one is in good shape to prove local injectivity. While this idea is simple to grasp, implementing it requires some technology and this has to be designed and tailored to the problem at hand.
Let us describe the ``technology'' that works for our problem.

Given an arbitrary real number $\digamma>0$, we define the following localized operator (in the basis $\{{\bf e}_1,\cdots, {\bf e}_d\}$): 
\begin{equation*}
N^\digamma[f, \beta](x,y)=(N^\digamma_{1}[f,\beta](x,y),\cdots,N^\digamma_{d}[f,\beta](x,y)),
\end{equation*}
where
\begin{align*}
& N^\digamma_{j}[f,\beta](x,y)\\
:=& e^{-\digamma/x} \int g_{sc}(\lambda\p_x+\omega\p_y) W^* (x,y,\lambda,\omega,{\bf e}_j)\chi_j\left(\frac{\lambda}{x}\right)\Big(I_we^{\digamma/x}[f,\beta]\Big)(x,y,\lambda,\omega,{\bf e}_j)\,d\lambda d\omega.
\end{align*}

Let us explain what each item is trying to achieve in this definition.
\begin{enumerate}

\item $g_{sc}$ denotes the scattering metric. In our local coordinates, we can consider them as
$$g_{sc}:=x^{-4}dx^2+x^{-2}h,$$
where $h$ is the metric on the level sets of $x$. The scattering metric $g_{sc}$ is not assumed to have any relation with the underlying metric $g$.
Using $g_{sc}$ we can convert vectors into co-vectors and hence we can make $N^\digamma[f,\beta]$ into a 1-form, so $N^\digamma[f,\beta]$ is a $(n\times d)$-matrix valued 1-form (as $f$). 
\item The functions $\chi_j\in C^{\infty}_c(\mathbb R),\, j=1,\cdots, d$ are even cut-off functions, so that for $0< x\ll 1$ only those $I_w e^{\digamma/x}[f,\beta](x,y,\lambda,\omega,\xi)$ with $\lambda$ sufficiently small will contribute to $N^\digamma [f,\beta]$. We further allow $\chi_j$ to depend smoothly on $y$ and $\omega$, and we can consider 
$$\chi_j(x,y,\lambda,\omega)=\chi_j\left(\frac{\lambda}{x},y,\omega\right)$$
as a smooth function on $S\Omega$.
Essentially the cut-off functions ensure that only geodesics belonging to a small cone around each point are being considered.
\item The adjoint $W^*$ of the weight $W$ is the correct object if $N^\digamma$ is going to mimick the normal operator associated to $I_{w}$ at least if we only pay attention to the first component of the pair $[f,\beta]$.
\item Conjugation by the exponential weights $ e^{-\digamma/x}$ will ensure that our operator will eventually belong to the scattering algebra. This will produce exponentially weak estimates as we approach the artificial boundary but it will be of no concern.

\end{enumerate}

An important feature of the construction is its dependence on the parameter $c$ as we will be able to attain injectivity only for small $c$. In this respect our situation will be no different to that in \cite[Section 2.5]{UV16}.
We now proceed to study the operator $N^{\digamma}$ in detail.

\subsection{Scattering calculus and Schwartz kernels}
In this subsection we introduce the elements needed from the scattering calculus and we compute various Schwartz kernels to show that $N^{\digamma}$ fits into this calculus. This subsection is not selfcontained and relies considerably on \cite{UV16,SUV14}. Section 2 in \cite{UV16} contains very relevant background on the scattering calculus and the original reference is \cite{Me94}.

For each fixed ${\bf e}_j$, the maps
$$\Gamma_+: S\tilde M\times [0,\infty)\rightarrow [\tilde M\times\tilde M; \diag],\,\;\; \Gamma_+(z,v,t)=(z, |z'-z|, \frac{z'-z}{|z'-z|})$$
and
$$\Gamma_-: S\tilde M\times (-\infty,0]\rightarrow [\tilde M\times\tilde M; \diag],\,\;\;, \Gamma_-(z,v,t)=(z, -|z'-z|, -\frac{z'-z}{|z'-z|})$$
are two diffeomorphisms near $S\tilde M\times \{0\}$. Here $z'=(x',y')=\gamma_{z,v,{\bf e}_j}(t)$, $[\tilde M\times\tilde M; \diag]$ is the {\it blow-up} of $\tilde M$ at the diagonal $z=z'$. At first glance, the definition of $z'=\gamma_{z,v,{\bf e}_j}(t)$ depends on ${\bf e}_j$, we might need to denote it by $z'_j$. However, one can introduce an additional (local) diffeomorphism $\phi_j$ near $z$ such that, if $z'=\gamma_{z,v}(t)$ stands for the usual geodesic, $\phi_j(z)=z$ and $\phi_j(z')=z'_j$ for $|t|$ small. Notice that since the Jacobian of $\phi_j$ is the identity at $z$, we can use this uniform parameter $z'$ from now on for any $\xi_j$. Similar to \cite{UV16}, we can use $(x,y,|y'-y|,\frac{x'-x}{|y'-y|},\frac{y'-y}{|y'-y|})$ as the local coordinates on $\Gamma_+(\supp\chi_j\times[0,\delta))$, and analogously for $\Gamma_-(\supp\chi_j\times (-\delta,0])$ the coordinates are $(x,y,-|y'-y|,-\frac{x'-x}{|y'-y|},-\frac{y'-y}{|y'-y|})$, when $x$ is close to $0$, for $\delta>0$ sufficiently small. 

As we want to study the microlocal properties of $N^\digamma$ up to the so-called scattering front face $x=0$, we apply the scattering coordinates $(x,y,X,Y)$ from \cite{UV16}, where 
$$X=\frac{x'-x}{x^2}, \, Y=\frac{y'-y}{x}.$$
Under the scattering coordinates
$$dt\, d\lambda\, d\omega=x^2|Y|^{1-n}J(x,y,X,Y)\,dXdY$$
with $J|_{x=0}=1$.
Notice that $|z'-z|\sim |y'-y|$ and we have an asymptotic expansion of $t$ as
$$t=|z'-z|+O(|z'-z|^2)=|y'-y|+O(|y'-y|^2)=x|Y|+O(x^2).$$
Let 
$$\alpha_j(z,v,t)=\frac{d^2 x}{dt^2}(\gamma_{z,v,{\bf e}_j}(t)),$$
in particular $\alpha_j(0,y,0,\omega,0)>0$ for any $\omega\in \mathbb S^{n-2}$ by the convexity assumption. 
Let $^{sc}T^*\overline\Omega$ be the scattering cotangent bundle over $\overline\Omega$ whose basis is $\{\frac{dx}{x^2}, \frac{dy}{x}\}$ with $\frac{dy}{x}$ a short notation for $(\frac{dy_1}{x},\cdots,\frac{dy_{n-1}}{x})$, while the dual space (scattering tangent bundle) is $^{sc}T\overline\Omega$ with the basis $\{x^2\p_x,x\p_y\}$. Let $(x',y',\lambda',\omega',\xi')$ be the short notation for $(x(t),y(t),\lambda(t),\omega(t),\xi(t))$. It was shown in \cite{UV16, SUV14} that the following asymptotic expansions in the scattering tangent and cotangent bases hold (for fixed ${\bf e}_j$):
\begin{equation}\label{lambda omega}
\begin{split}
& g_{sc}\Big((\lambda\circ\Gamma_{\pm}^{-1})\p_x+(\omega\circ\Gamma_{\pm}^{-1})\p_y\Big)\\
= & x^{-1}\bigg(\Big(\pm\frac{X-\alpha_j(x,y,,\pm \frac{xX}{|Y|},\pm \hat Y,\pm x|Y|)|Y|^2}{|Y|}+x\tilde\Lambda_\pm(x,y,\frac{x
  X}{|Y|},\hat Y,x|Y|)\Big)\,\frac{dx}{x^2}\\
&\qquad\qquad+\Big(\pm\hat Y+x|Y|\tilde\Omega_\pm(x,y,\frac{x
  X}{|Y|},\hat Y,x|Y|)\Big)\,\frac{h(\p_y)}{x}\bigg)
\end{split}
\end{equation}
and
\begin{equation}\label{lambda omega prime}
\begin{split}
& (\lambda'\circ\Gamma_{\pm}^{-1})\p_x+(\omega'\circ\Gamma_{\pm}^{-1})\p_y\\
= & x^{-1}\bigg(\Big(\pm \frac{X+\alpha_j(x,y,\pm \frac{xX}{|Y|},\pm \hat
  Y,\pm x|Y|)|Y|^2}{|Y|}+x|Y|^2\tilde\Lambda'_\pm(x,y,\frac{x
  X}{|Y|},\hat Y,x|Y|)\Big)\,x^2\p_x\\
&\qquad\qquad+\Big(\pm \hat Y+x|Y|\tilde\Omega'_\pm(x,y,\frac{x
  X}{|Y|},\hat Y,x|Y|)\Big)\,x\p_y\bigg).
\end{split}
\end{equation}
Here $\tilde\Lambda_\pm$, $\tilde\lambda'_\pm$, $\tilde\Omega_\pm$ and $\tilde\Omega'_\pm$ are smooth functions and
$$\frac{xX}{|Y|}=\frac{x'-x}{|y'-y|},\quad \hat Y=\frac{y'-y}{|y'-y|},\quad x|Y|=|y'-y|.$$ 
From now on we denote $\alpha_j(0,y,0,\pm \hat Y,0)$ by $\alpha^\pm_j$ and $\frac{X-\alpha^\pm_j |Y|^2}{|Y|}$ by $S^\pm_j$, so $\frac{X+\alpha^\pm_j|Y|^2}{|Y|}=S^\pm_j+2\alpha^\pm_j |Y|$. On the other hand,
\begin{equation}\label{xi prime}
\xi'=\xi+t \dot\xi +O(t^2)=\xi+x |Y|\dot\xi +O(x^2). 
\end{equation}

Notice that in the scattering coordinates, $\chi_j$, $j=1,\cdots, d$, are smooth down to $x=0$. Combining the definition of $N^\digamma$, \eqref{lambda omega}, \eqref{lambda omega prime} and \eqref{xi prime}, one can easily calculate the Schwartz kernel of $N^\digamma$. 
In particular, we are interested in the behavior of the kernel at the scattering front face $x=0$. 

\begin{lemma}\label{kernel at x=0}
The Schwartz kernel of $N_j^\digamma$, denoted by $K_j^\digamma$, at the scattering front face $x=0$ has the following expression (in the scattering bases)
\begin{align*}
K_j^\digamma & (y,X,Y)=\\
& e^{-\digamma X}|Y|^{-n+1}\sum_{\diamond=+,-}\chi_j(S^\diamond_j)\begin{pmatrix} S^\diamond_j \\[.5em] \hat Y\end{pmatrix}\begin{pmatrix} W^* W & W^* Q\end{pmatrix} (0,y,0,\hat Y,{\bf e}_j)\begin{pmatrix} S^\diamond_j+2\alpha_j^\diamond|Y| & \hat Y\end{pmatrix}\,{\bf e}_{j}.
\end{align*}
\end{lemma}
For the sake of simplicity, we have used matrix block notations in the lemma, see also \cite{SUV14} for the case of unweighted ray transforms. In particular,
\begin{equation*}
\begin{split}
\begin{pmatrix} S+2\alpha |Y| & \hat Y\end{pmatrix}\quad &\mbox{stands for} \quad (S+2\alpha |Y|)(x^2\p_x)+\hat Y\cdot (x\p_y),\\
\begin{pmatrix} S \\[.5em] \hat Y\end{pmatrix}\quad &\mbox{stands for} \quad S \frac{dx}{x^2}+\hat Y\cdot \frac{dy}{x},
\end{split}
\end{equation*}
both terms are applied to the entries of $f$ and $\beta$ (they are indeed scalar). This says that the Schwartz kernel of $N^\digamma$ is a matrix valued distribution, whose entries are 1-1 tensors.


We decompose the operator $N^\digamma$ into two parts, so that 
$$N^\digamma [f,\beta]=N^1 f+ N^2 \beta.$$
Using Lemma \ref{kernel at x=0} it is easy to check that column vectors for the Schwartz kernels of $N^1$ and $N^2$ at the scattering front face $x=0$ are
\begin{equation}\label{kernel of N^1}
 e^{-\digamma X}|Y|^{-n+1}  \sum_{\diamond=+,-}\chi_j(S^\diamond_j) \begin{pmatrix} S^\diamond_j \\[.5em] \hat Y\end{pmatrix} W^* W(0,y,0,\hat Y,{\bf e}_j)\begin{pmatrix} S^\diamond_j+2\alpha_j^\diamond|Y| & \hat Y\end{pmatrix} {\bf e}_j
\end{equation}
and 
\begin{equation}\label{kernel of N^2}
 e^{-\digamma X}|Y|^{-n+1}  \sum_{\diamond=+,-}\chi_j(S^\diamond_j) \begin{pmatrix} S^\diamond_j \\[.5em] \hat Y\end{pmatrix} W^* Q(0,y,0,\hat Y,{\bf e}_j)\begin{pmatrix} S^\diamond_j+2\alpha_j^\diamond|Y| & \hat Y\end{pmatrix} {\bf e}_j
\end{equation}
respectively. In particular, $N^1$ and $N^2$, therefore $N^\digamma$, are scattering pseudodifferential operators of order $(-1,0)$ in $\overline\Omega$ in the sense of Melrose's scattering calculus \cite{Me94}; the set of such operators is denoted by $\Psi_{sc}^{-1,0}(\overline\Omega)$ or simply $\Psi_{sc}^{-1,0}$ (see \cite[Section2]{UV16} for more background). Generally the Schwartz kernel of a scattering pseudodifferential operator has the form $x^\ell K$ with non-zero $K$ smooth in $(x,y)$ down to $x=0$. For our case, the zero in the superscript of $\Psi^{-1,0}_{sc}$ means exactly that $\ell=0$, while the number $-1$, related to $K$, is simply the order as in the case of standard pseudodifferential operators.

\subsection{Ellipticity} In this subsection we focus on the ellipticity properties of the operator $N^{1}$.
In order to obtain ellipticity we need to narrow a bit the domain of $N^1$.  The previous analysis
was done assuming that $f$ was a $(n\times d)$-matrix valued 1-form; explicitly
\[f=(f^{\alpha}_{ij}dz^{j})_{n\times d}.\]
However, for our concrete problem we have additional structure, namely $f_{ij}=-f_{ji}$, so that $f$ really arises from a $\lg$-valued 2-form. On the other hand there is no reason for $N^1f$ to originate from
a $\lg$-valued 2-form, so to obtain an operator mapping between sections of the same vector bundle, we just
antisymmetrize $N^1f$. This additional algebraic operation is harmless for our purposes as it amounts to composing with a pseudodifferential operator of order zero and hence it will be implicitly assumed in the sequel, but see Remark \ref{remark:anti} below.
Let $^{sc}\Lambda^{2}(\bar{\Omega},\lg)$ denote the bundle of scattering 2-forms in $\overline{\Omega}$ with values in $\lg$.

\begin{prop}\label{N^1 elliptic}
There exist $\chi_j\in C^\infty_c(\mathbb R)$ with $\chi_j(0)>0$, $\chi_j\geq 0$, $j=1,\cdots,d$, such that $N^1\in \Psi^{-1,0}_{sc}(\overline{\Omega}, ^{sc}\Lambda^{2}(\overline{\Omega},\lg),^{sc}\Lambda^{2}(\overline{\Omega},\lg))$ is elliptic acting on Lie algebra valued 2-forms.
\end{prop}

\begin{proof}
We are interested in the asymptotic behaviour, which is uniform in $c$, of the principal symbol of $N^1$ on $(z,\rho)\in\, ^{sc}T^*\overline\Omega$ for $z=(x,y)$, $x\geq 0$ small and $\rho=(\zeta,\eta)$. Note that $\overline\Omega$ is a manifold with boundary, hence there are two types of behaviours, i.e. at finite points and fibre infinity of $^{sc}T^*\overline\Omega$. Since the Schwartz kernel of $N^1$ is smooth in $(x,y)$ up to the front face $x=0$ and the Lie algebra valued 2-form is supported in $x<c$ for $c$ sufficiently close to $0$, it is enough to consider the symbol at $x=0$.

We start with the case of fibre infinity of the scattering cotangent bundle. Let $\chi_j\in C^\infty_c(\mathbb R)$ satisfy $\chi_j\geq 0$ and $\chi_j>0$ near $0$. 
One may ignore the extra decay factor $|Y|$ in the Schwartz kernel. 
So let $\tilde S=X/|Y|$, $|\rho|$ sufficiently large; by \eqref{kernel of N^1} each column vector of the principal symbol $\sigma_p(N^1)$ at $\rho=(\zeta,\eta)$ has the form (up to some positive constant)
\begin{align*}
|\rho|^{-1} \Big(\int_{\rho^\perp\cap (\mathbb R\times \mathbb S^{n-2})} \chi_j(\tilde S)\begin{pmatrix}\tilde S \\[.5em] \hat Y\end{pmatrix}(W^*W)(0,y,0,\hat Y,{\bf e}_j)\begin{pmatrix} \tilde S & \hat Y\end{pmatrix} \,d\tilde S d\hat Y \Big) {\bf e}_j.
\end{align*}
Here $\rho^\perp$ is the orthogonal complement of $\rho$. Then
\begin{align*}
\Big(\sigma_p(N^1)& (\rho)\varpi,\varpi\Big)=\\
& |\rho|^{-1}\sum_{j=1}^d \int_{\rho^\perp\cap (\mathbb R\times \mathbb S^{n-2})} \chi_j(\tilde S)\left\vert W(0,y,0,\hat Y,{\bf e}_j)( \varpi_{x} \tilde S+ \varpi_{y}\cdot \hat Y)\cdot {\bf e}_j\right\vert ^2\,d\tilde S d\hat Y.
\end{align*}
Now if $\Big(\sigma_p(N^1) (\rho)\varpi,\varpi\Big)=0$,
we get that 
\begin{equation}\label{zero integrand 1}
W(0,y,0,\hat Y,{\bf e}_j)( \varpi_{x} \tilde S+ \varpi_{y}\cdot \hat Y)\cdot {\bf e}_j=0,\quad j=1,\cdots d,
\end{equation} for $\tilde S$ small and $\tilde S \zeta+\hat Y\cdot \eta=0$. Note that $\varpi$ is a Lie algebra valued 2-form that can be written as a matrix valued 1-form:
$$\varpi=(\varpi_x,\, \varpi_y)\quad \mbox{with}\quad \varpi_x=(\varpi^\alpha_{ix})_{n\times d},\quad  \varpi_y=(\varpi^\alpha_{iy})_{n\times d}.$$
Since $W(0,y,0,\hat Y,{\bf e}_j)$ is invertible for any $j=1,\cdots,d$, and $\{{\bf e}_1,\cdots,{\bf e}_d\}$ form a basis, \eqref{zero integrand 1} implies that
\begin{equation*}
\varpi_x \tilde S+\varpi_y\cdot \hat Y =0
\end{equation*}
for any $(\tilde S,\hat Y)\in \rho^\perp\cap (\mathbb R\times \mathbb S^{n-2})$ with $\tilde S$ close to zero. Notice that in dimension $\geq 3$, the set of such $(\tilde S,\hat Y)$ (i.e. small $|\tilde S|$) always spans $\rho^\perp$, which means 
that $(\varpi^\alpha_{ix}, \varpi^\alpha_{iy})$ is parallel to $\rho$ for any $\alpha$ and $i=x,y^1,\cdots,y^{n-1}$, i.e. there exist $u^\alpha=(u^\alpha_1,\cdots,u^\alpha_n)$ such that $(\varpi^\alpha_{ix},\varpi^\alpha_{iy})=u^\alpha_i \rho$ (this is where the dimension assumption takes effect; when dim $M=2$, the set of $(\tilde S,\hat Y)$ satisfying the requirements could be empty if $|\zeta|$ is relatively small comparing to $|\eta|$, for $\rho=(\zeta,\eta)$). 


As a Lie algebra valued 2-form, $\varpi(v,v)=0$ for any vector $v$. Thus $(\rho\cdot v)(u^\alpha\cdot v)=0$, $\forall v,\,\alpha$. Let $v=\rho$, as $\rho\neq 0$, we get that $u^\alpha\cdot \rho=0$, i.e. $u^\alpha\in\rho^\perp$. On the other hand, if $v=\rho+u^\alpha$, then $|u^\alpha|^2=0$, i.e. $u^\alpha=0$. This shows that $\varpi=0$, so $N^1$ is elliptic at fibre infinity.



To analyze the (scattering) principal symbol of $N^1$ at finite points, in particular for $\rho$ close to $0$, of $^{sc}T^*\overline\Omega$, we take the $(X,Y)$-Fourier transform of \eqref{kernel of N^1}. Following the strategy in \cite{UV16}, we first calculate the case when $\chi_j$, $j=1\cdots,d$ are of Gaussian type, so let $\chi_j(s)=e^{-s^2/(2\digamma^{-1}\alpha_j)}$. 
A simple calculation gives that each column vector of the scattering principal symbol $\sigma_{sc}(N^1)$ at $\rho=(\zeta,\eta)$ has the form 
\begin{align*}
\frac{C}{\sqrt{\zeta^2+\digamma^2}}\times \int_{\mathbb S^{n-2}}\begin{pmatrix} -\overline{\kappa} \hat Y\cdot \eta \\[.5em] \hat Y\end{pmatrix}(W^*W)(0,y,0,\hat Y,{\bf e}_j)\begin{pmatrix} -\kappa \hat Y\cdot \eta & \<\hat Y,\cdot\>\end{pmatrix}e^{-\frac{|\hat Y\cdot \eta|^2}{2\digamma^{-1}(\zeta^2+\digamma^2)\alpha_j}}\,d\hat Y \, {\bf e}_j,
\end{align*} 
where $\kappa=\frac{\zeta-i\digamma}{\zeta^2+\digamma^2}$.

Consider 
\begin{align*}
\Big(&\sigma_{sc}(N^1) (\zeta,\eta)\varpi,\varpi\Big)=\\
& \frac{C}{\sqrt{\zeta^2+\digamma^2}}\sum_{j=1}^d \int_{\mathbb S^{n-2}} \left\vert W(0,y,0,\hat Y,{\bf e}_j) \Big(\varpi_{y}\cdot \hat Y-\varpi_{x} \frac{\zeta-i\digamma}{\zeta^2+\digamma^2}\hat Y\cdot \eta\Big) \cdot {\bf e}_j\right\vert ^2 e^{-\frac{|\hat Y\cdot \eta|^2}{2\digamma^{-1}(\zeta^2+\digamma^2)\alpha_j}}\,d\hat Y.
\end{align*}
If $\Big(\sigma_{sc}(N^1) (\zeta,\eta)\varpi,\varpi\Big)=0$, then similar to the case at fibre infinity we have that
\begin{equation*}
\varpi_{y}\cdot \hat Y-\varpi_{x} \frac{\zeta-i\digamma}{\zeta^2+\digamma^2}\hat Y\cdot \eta=0,\quad \forall \, \hat Y\in \mathbb S^{n-2}.
\end{equation*}
This implies that $(\varpi^\alpha_{ix},\varpi^\alpha_{iy})$ is parallel to $\mu=(\zeta+i\digamma,\eta)$, $\forall i,\,\alpha$. Now use the same idea as in the fibre infinity case ($\varpi$ is a 2-form) to get that there exists a vector $u^\alpha$ such that $(u^\alpha\cdot v)(\mu\cdot v)=0$, $\forall v,\,\alpha$. Let $v=\overline\mu$, notice that $\digamma>0$, we have $u^\alpha\cdot\overline\mu=0$. Next, take $v=\overline{\mu+u^\alpha}$, we achieve that $|u^\alpha|^2=0$, i.e. $u^\alpha=0$. Therefore $\varpi=0$ and $N^1$ is elliptic at finite points for Gaussian type $\chi_j$.

Notice that a Gaussian type $\chi_j$ is not compactly supported. Nevertheless one can take $\chi_j^k(s)=\phi(s/k)\chi_j(s)$ for some $\phi\in C^\infty_c(\mathbb R)$ with $\phi\geq 0$ and $\phi(0)>0$. For each $j$, the Schwartz kernel of $N^1_j(\chi^k_j)$ converges to the one of $N^1_j(\chi_j)$ in the space of distributions as $k\to \infty$, hence we also have convergence for the principal symbols. Since the dimension $d$ of the Lie algebra is finite, for large enough $k$, $N^1=(N^1_1(\chi^k_1),\cdots,N^1_d(\chi^k_d))$ is elliptic too.

Combining the two cases, we establish the ellipticity of $N^1$ acting on 2-forms for properly chosen $\chi_j$, $j=1,\cdots,d$.
\end{proof}



\begin{rem} \label{remark:anti} As we pointed out just before stating Proposition \ref{N^1 elliptic}, we need to antisymmetrize $N^{1}$ in order to make it map the set of antisymmetric matrix-valued 1-forms to itself. Strictly speaking, if $A$ denotes the antisymmetrization operator, we are required to prove that $(\sigma(A\circ N^{1}) \varpi,\varpi)\neq 0$ in order to establish ellipticity. However, since the inner product that we are working with pairs symmetric and antisymmetric matrices to zero, it is enough to show that  $(\sigma(N^{1}) \varpi,\varpi)\neq 0$ as done above.

\end{rem}

Now we move back to the operator $N^\digamma$. We denote $H_{sc}^{s,r}$ the {\it scattering Sobolev spaces}, which are locally equivalent to the standard weighted Sobolev spaces $H^{s,r}(\mathbb R^n)=\<z\>^{-r}H^s(\mathbb R^n)$, see \cite[Section 2]{UV16} for details. By Proposition \ref{N^1 elliptic} and the local nature of the problem (the error term in the elliptic estimate, which is proportional to $c$, can be absorbed), as in \cite{UV16}, we have the following corollary.

\begin{cor}\label{elliptic estimate}
For $c>0$ sufficiently small and properly chosen $\chi_j\in C^{\infty}_c(\mathbb R)$, $j=1,\cdots d$, given any pair of matrix valued 1-forms $[f,\beta]$, where the dimensions of $f$ and $\beta$ are $n\times d$ and $d\times d$ respectively, with $f$ induced by a Lie algebra valued 2-form, then
$$\|f\|_{H^{s,r}_{sc}(\overline\Omega)}\leq C\big(\|N^\digamma[f,\beta]\|_{H^{s+1,r}_{sc}(\overline\Omega)}+\|\beta\|_{H^{s,r}_{sc}(\overline\Omega)}\big).$$
\end{cor}
(Note that the definition of the neighborhood $\Omega$ depends on the small parameter $c>0$.)

In Corollary \ref{elliptic estimate} and up to now, $f$ and $\beta$ are independent of each other. However, to achieve the main result of this section, the injectivity of $I_w$, we need to use the relation between the connection $A$ and its curvature $F$. This will be the subject of the next (and final) subsection and we shall exploit this relationship in the {\it normal gauge}.

\subsection{The normal gauge and the proof of local injectivity}

Before continuing the argument, we do gauge transforms of $A$ and $\tilde A$ in the $(x,y)$ coordinates. Observing that $A=A_xdx+A_y dy$ as a 1-form, let $\Omega'$ be some open set (still in the $(x,y)$ coordinate system) with $\overline\Omega\subset \Omega'$, we construct $u:\Omega'\to G$ such that
$$\p_x u+A_x u=0,\quad u|_{\Omega' \cap \p M}=e,$$
and define $A'=u^{-1}du+u^{-1}A u$. Then it is easy to see that $A'_x=0$. Similarly, one constructs $\tilde u:\Omega'\to G$ for $\tilde A$ and $\tilde A'=\tilde u^{-1}d\tilde u+\tilde u^{-1} \tilde A \tilde u$ with $\tilde A'_x=0$. Since $A=\tilde A$ in $\Omega'\setminus M^o$ (we denote the interior of $M$ by $M^o$), we get that $u=\tilde u$ in the same set,  thus $A'=\tilde A'$ and $F'=\tilde F'$ in $\Omega'\setminus M^o$ accordingly. After doing the gauge transforms, we only need to consider $A$ and $\tilde A$ satisfying the {\it normal gauge condition} (i.e. $A_x=\tilde A_x=0$), and the resulting $[\F,\hat c\A]$ is still (locally near $p$) supported in $M$. 

Let us do a bit more analysis on the pair $[\F,\hat c\A]$ under the normal gauge. Since $A_x=\tilde A_x=0$, $\A_x=0$ too. On the other hand, $F=dA+A\wedge A$ as a 2-form has the following expression in local coordinates
\begin{equation*}
F_{xx}=0,\quad F_{xy}=\p_x A_y,\quad F_{yx}=-\p_x A_y,\quad F_{yy}=p(A_y,\p_y A_y),
\end{equation*}
for some function $p$. Thus for $\F=F-\tilde F$,
\begin{equation*}
\F_{xx}=0,\quad \F_{xy}=\p_x \A_y,\quad \F_{yx}=-\p_x \A_y,\quad \F_{yy}=p(A_y,\p_y A_y)-p(\tilde A_y,\p_y \tilde A_y).
\end{equation*}
We need the following estimate in the scattering Sobolev spaces. 

\begin{lemma}\label{stronger control}
Given any $\delta>0$, there exists $\varepsilon>0$ such that for $0<c<\varepsilon$, if $\A$ is supported in $\overline\Omega \cap M$, then
$$\|e^{-\digamma/x}\A\|_{H^{s,r}_{sc}(\overline\Omega)}\leq \delta \|e^{-\digamma/x}\F\|_{H^{s,r}_{sc}(\overline\Omega)}.$$
\end{lemma}

\begin{proof}
By a simple calculation,
$$-ix^2 e^{-\digamma/x}\p_x \A_y=D^\digamma_x e^{-\digamma/x} \A_y,$$
where $D^\digamma_x=e^{-\digamma/x}(-x^2 i\p_x) e^{\digamma/x}$ with principal symbol $\zeta+i\digamma$. In particular, one obtains the Fredholm property in the scattering Sobolev norms,
$$\|e^{-\digamma/x}\A_y\|_{H^{s,r}_{sc}(\overline\Omega)}\leq C (\|x^2 e^{-\digamma/x}\p_x \A_y\|_{H^{s,r}_{sc}(\overline\Omega)}+\|e^{-\digamma/x}\A_y\|_{H^{-K,-L}_{sc}(\overline\Omega)})$$
for any $s,r,K,L$. Moreover, since $D^\digamma_x$ has trivial kernel acting on distributions supported in $\overline \Omega \cap M$, the second term on the right hand side of the inequality can be dropped by a standard functional analysis argument, see e.g. \cite[Proposition V.3.1]{Ta81}. Notice that $\A$ is supported in $\{0\leq x\leq c\}$; we have 
$$\|e^{-\digamma/x}\A_y\|_{H^{s,r}_{sc}(\overline\Omega)}\leq C \varepsilon^2\|e^{-\digamma/x}\p_x \A_y\|_{H^{s,r}_{sc}(\overline\Omega)},$$
thus 
$$\|e^{-\digamma/x}\A\|_{H^{s,r}_{sc}(\overline\Omega)}\leq C \varepsilon^2\|e^{-\digamma/x}\F\|_{H^{s,r}_{sc}(\overline\Omega)}.$$
Taking $\varepsilon$ sufficiently small, this proves the lemma.
\end{proof}

Taking into account the exponential weights in the definition of $N^\digamma$, Corollary \ref{elliptic estimate} and Lemma \ref{stronger control}, we have that for $[\F,\hat c\A]$ supported in $\{x\leq c\}$, $c>0$ sufficiently small
\begin{equation*}
\|e^{-\digamma/x}\F\|_{H^{s,r}_{sc}(\overline\Omega)}\leq \frac{C}{1-C'\delta}\|N^\digamma e^{-\digamma/x}[\F,\hat c\A]\|_{H^{s+1,r}_{sc}(\overline\Omega)}.
\end{equation*}
In particular, this proves the local invertibility of $I_w$.

\begin{thm}\label{local invertibility}
Given potentials $A$ and $\tilde A$ in the normal gauge, if $I_w[\F,\hat c\A](z,v,\xi)=0$ for any $(z,v)$ close to $S_p\p M$ and $\xi\in \mathcal O$, then $\F$ and therefore $\A$, vanish near $p$ in $M$.
\end{thm}

Now we are ready to prove the local injectivity of the lens rigidity problem.

\begin{proof}[Proof of Theorem \ref{local thm}]
Given Yang-Mills potentials $A$ and $\tilde A$, by the assumptions and the analysis from Sections \ref{section:local} and \ref{sec:injectivity of I_w}, there exist maps $u,\, \tilde u:U\to G$, where $U$ is an open neighborhood of $p$ in $M$, 
and $u|_{U\cap \p M}=\tilde u|_{U\cap \p M}=e$, such that if $A':=u^{-1}du+u^{-1}Au$, $\tilde A':=\tilde u^{-1}d\tilde u+\tilde u^{-1}\tilde A\tilde u$, then $A'$ and $\tilde A'$ are in the normal gauge under the local coordinates $(x,y)$ with respect to the hypersurface $H$ defined in Section \ref{section:local}. Moreover $A'$ and $\tilde A'$ have the same boundary jet near $p$, thus they can be extended identically near $p$ into $\tilde M$. Since $A$ and $\tilde A$ induce the same lens data near $S_p\p M$, so do $A'$ and $\tilde A'$. Applying the integral identity \eqref{integral identity 3}, by Theorem \ref{local invertibility}, we get that $A'$ and $\tilde A'$ are indeed identical near $p$ in $M$. This implies that if we denote $w=u\tilde u^{-1}$, then $\tilde A=w^{-1}dw+w^{-1}Aw$ in $M$ near $p$ with $w=e$ on $\p M$ close to $p$, which proves the local lens rigidity.
\end{proof}


\section{Global lens rigidity: proof of Theorem \ref{thm:main}}
\label{section:proofglobal}

In this final section we show how to use the local Theorem \ref{local thm} together with the strictly YM-convex function $f$ and the boundary $\partial M$ to prove the global result from the introduction.

As we already remarked in the introduction $f$ is also strictly convex as a function on the compact connected Riemannian manifold $(M,g)$, hence it has some properties that we now summarize; for proofs we refer for instance to \cite[Section 2]{PSUZ16}. The function $f$ has a unique local minimum point $z_{0}$ in $M$
and $f$ attains its global minimum there. Moreover, the set of critical points of $f$ is either $z_0$ or the empty set.
Hence, without loss of generality we assume that there exist $\tau<0$ and $z_0\in M$ such that $\inf_{z\in M} f(z)=\tau$, $f^{-1}(\tau)=\{z_0\}$, and $M= f^{-1}([\tau,0])$. We denote $U_\tau=\{f>\tau\}=M\setminus \{z_0\}$.

Observe that the level set function $x$ introduced in Section \ref{section:local} is locally defined, and depends on the convex point $p$. To prove the global lens rigidity, we apply a layer stripping argument similar to the one in \cite{PSUZ16, SUV17}. In particular, the level sets of $f$ will take the role of the boundary $\p M$ when we move the argument into the interior of $M$.

\begin{proof}[Proof of Theorem \ref{thm:main}] The proof has two parts. First we will construct a gauge on the set $U_{\tau}$ taking $A$ to $\tilde{A}$. For this we will need to glue the local gauges coming from Theorem \ref{local thm} and we will follow the strategy in \cite[Proposition 6.2]{PSUZ16} where a similar gluing was carried out for a related a linear problem.
The second part of the proof consists in showing that the gauge constructed on $U_{\tau}$ extends to $M$. For this we will exploit the fact that our gauges take values in the {\it compact} Lie group $G$.
Hence we start with:

{\it Claim:} There exists a smooth $u: U_\tau\to G$ with $u|_{\p M\cap U_\tau}=e$, such that $\tilde A=u^{-1}du+u^{-1}A u$ in $U_\tau$.

Given any $\tau<t\leq 0$, $f^{-1}(t)$ is a compact set (hypersurface), which is strictly YM-convex with respect to $A$ by assumption. Clearly $f^{-1}(0)\subset \p M$ since $f$ cannot attain a maximum at an interior point. Hence for any $p\in f^{-1}(0)$, by Theorem \ref{local thm} there exist a neighborhood $O_p$ of $p$ in $M$ and $u_p:O_p\to G$ with $u_p|_{\p M\cap O_p}=e$ such that $\tilde A=u_p^{-1}du_p+u_p^{-1} A u_p$ in $O_p$. On the other hand, for any $p,q \in f^{-1}(0)$, if $O_p\cap O_q\neq \emptyset$, we have that
$$u_p^{-1}du_p+u_p^{-1} A u_p=u_q^{-1}du_q+u_q^{-1} A u_q\,\,\mbox{in}\,\, O_p\cap O_q,\quad u_p|_{\p M\cap O_p\cap O_q}=u_q|_{\p M\cap O_p\cap O_q}=e.$$
Equivalently
\begin{equation}\label{ODE}
d(u_p u_q^{-1})+A(u_p u_q^{-1})-(u_p u_q^{-1})A=0, \quad u_p u_q^{-1}|_{\p M\cap O_p\cap O_q}=e.
\end{equation}
Notice that for any point $z\in O_p\cap O_q$, we can find a curve $\gamma$ connecting $z$ with $\p M\cap O_p\cap O_q$, then \eqref{ODE} reduces to an ODE along $\gamma$ with initial condition $e$. This implies that $u_q^{-1}u_p\equiv e$, i.e. $u_p=u_q$, in $O_p\cap O_q$. By the compactness of $f^{-1}(0)$, there exist $\epsilon>0$ and $u:\{f\geq -\epsilon\}\to G$, $u|_{\{f\geq -\epsilon\}\cap \p M}=e$ such that $\tilde A=u^{-1}du+u^{-1}A u$ in $\{f\geq -\epsilon\}$. 

Assume now that $\tilde A=u^{-1}du+u^{-1}Au$ in $\{f\geq s\}$ for some $s>\tau$, $u|_{\{f\geq s\}\cap \p M}=e$. We extend $u$ smoothly to $M$ such that $u|_{\p M}=e$ and still denote it by $u$. Since $A$ and $\tilde A$ have the same lens data, and the level set $f^{-1}(t)$ is strictly YM-convex with respect to both $A$ and $\tilde A$ for any $t\geq s$, it is not difficult to check that $\tilde A$ and $u^{-1}d u+u^{-1}A u$ have the same lens data near $f^{-1}(s)$ in $\{f\leq s\}$. We apply Theorem \ref{local thm} again to the compact set $f^{-1}(s)$ to obtain:  there exists $\epsilon>0$ and $v:\{s-\epsilon\leq f\leq s\}\to G$, $v|_{(\p M\cap \{s-\epsilon\leq f\leq s\})\cup f^{-1}(s)}=e$, such that 
\begin{equation}\label{extension beyond s}
\tilde A=v^{-1}dv+v^{-1}(u^{-1}du+u^{-1}Au)v
\end{equation}
in $\{s-\epsilon\leq f\leq s\}$. (To make the construction work, we indeed need the fact that $\p M$ is strictly YM-convex and that $f^{-1}(s)\cap \p M$ is compact, this is similar to the proof of \cite[Proposition 6.2]{PSUZ16}.) Using \eqref{extension beyond s} we can check that if we extend $v$ by $e$ to $\{f\geq s-\epsilon\}$, $v$ is smooth. If we define $w=uv$, so $w|_{\{f\geq s-\epsilon\}\cap \p M}=e$, then \eqref{extension beyond s} implies that $\tilde A=w^{-1}dw+w^{-1}Aw$ in $\{f\geq s-\epsilon\}$.


Notice that in above extension argument, the small constant $\epsilon=\epsilon(s)$ can be taken uniform for $s'$ close to $s$ (see also the discussion at the beginning of Section \ref{sec:injectivity of I_w}). We complete the proof of the {\it Claim} arguing by contradiction. Suppose
\[ s_0:=\inf \{t: A\;\text{and}\; \tilde A\; \text{are gauge equivalent in}\; f\geq t\}>\tau.\]
Then $\epsilon(s')=\epsilon(s_{0})$ for $s'$ close enough to $s_0$ by uniformity. The argument in the previous paragraph applied to $s'$ implies that $A$ and $\tilde A$ are gauge equivalent in $\{f\geq s'-\epsilon\}$, while $s'-\epsilon<s$, this is a contradiction.

So far we have shown that $A$ and $\tilde A$ are gauge equivalent in $U_{\tau}$ with some smooth $u:U_\tau\to G$, $u|_{U_\tau\cap \p M}=e$. The only thing left is to show that $u$ can be extended smoothly to $M$. 
By the gauge equivalence, we have that on $U_\tau$,
\begin{equation}\label{control du}
du=u\tilde A-Au.
\end{equation}
Since $A$ and $\tilde{A}$ are smooth in $M$ and $G$ is a compact Lie group, \eqref{control du} implies
that the norm of $du$ is uniformly bounded in $U_{\tau}$ and hence $u$ is uniformly continuous.
Since $M\setminus U_{\tau}=\{z_0\}$ is a single point,  this implies that $u$ can be extended 
continuously to $M$. We claim that $u$ is indeed smooth on $M$. 
Equation \eqref{control du} shows that $du$ can be continuously extended to $M$ too. Moreover, by differentiating both sides of \eqref{control du} repeatedly, we can extend any higher order derivative of $u$ continuously from $U_\tau$ to $M$. Now by essentially applying the fundamental theorem of calculus, one can show that the continuous extension $u$ is smooth on $M$, see e.g. \cite[Lemma 6.2]{Ce17}. In particular, \eqref{control du} holds on $M$ now, which shows the gauge equivalence of $A$ and $\tilde A$ on $M$.

This completes the proof of the global lens rigidity theorem.
\end{proof}


\end{document}